\documentclass[reqno]{amsart}
\usepackage{amsmath}
\usepackage{amssymb}
\usepackage{amsthm}
\usepackage{latexsym}
\usepackage{hyperref}
\usepackage{enumerate}
\usepackage{graphicx}
\usepackage[all]{xy}
\usepackage{tikz-cd} 
\usepackage{color}
\usepackage{verbatim}
\usepackage{mathtools}

\usepackage[margin=1in]{geometry}
\usepackage[mathscr]{euscript}

\newtheorem{thm}{Theorem}[section]
\newtheorem{cor}[thm]{Corollary}
\newtheorem{lem}[thm]{Lemma}
\newtheorem{prop}[thm]{Proposition}
\newtheorem{conj}[thm]{Conjecture}
\newtheorem{hyp}[thm]{Hypothesis}

\DeclareMathAlphabet\mathbfcal{OMS}{cmsy}{b}{n}
\theoremstyle{definition}
\newtheorem{defn}[thm]{Definition}
\newtheorem{rem}[thm]{Remark}

\newtheorem{ex}[thm]{Example}

\newcommand{\CC}{{\mathbb{C}}}

\newcommand{\bZ}{{\mathbb{Z}}}
\newcommand{\bR}{{\mathbb{R}}}

\newcommand{\Fq}{{\mathbb{F}_q}}

\newcommand{\GL}{{\mathsf{GL}}}

\newcommand{\red}{{\mathrm{red}}}

\newcommand{\SL}{{\mathsf{SL}}}
\newcommand{\Ind}{{\mathrm{Ind}}}

\newcommand{\cH}{{\mathcal{H}}}
\newcommand{\Hom}{{\mathrm{Hom}}}
\newcommand{\Gal}{{\mathrm{Gal}}}
\newcommand{\PGL}{{\mathsf{PGL}}}

\newcommand{\SO}{{\mathsf{SO}}}

\newcommand{\FT}{{\mathrm{FT}}}

\newcommand{\Irr}{{\mathrm{Irr}}}

\newcommand{\Nor}{{\mathrm{N}}}
\newcommand{\Cent}{{\mathrm{Cent}}}
\newcommand{\rZ}{{\mathrm{Z}}}

\newcommand{\bG}{{\mathbf{G}}}
\newcommand{\bH}{{\mathbf{H}}}

\newcommand{\cA}{{\mathcal{A}}}
\newcommand{\cC}{{\mathcal{C}}}

\newcommand{\cL}{{\mathcal{L}}}
\newcommand{\cG}{{\mathcal{G}}}
\newcommand{\cR}{{\mathcal{R}}}

\newcommand{\cY}{{\mathcal{Y}}}

\newcommand{\ellip}{{\mathrm{ell}}}

\newcommand{\St}{{\mathrm{St}}}

\newcommand{\Ad}{{\mathrm{Ad}}}

\newcommand{\der}{{\mathrm{der}}}
\newcommand{\ind}{{\mathrm{ind}}}

\newcommand{\InnT}{{\mathrm{InnT}}}

\newcommand{\EP}{{\mathrm{EP}}}
\newcommand{\Ext}{{\operatorname{Ext}}}
\newcommand{\tr}{{\mathrm{tr}}}

\newcommand{\sgn}{{\mathrm{sgn}}}

\newcommand{\un}{{\mathrm{un}}}
\newcommand{\Frob}{{\mathrm{Frob}}}

\newcommand{\al}{{\alpha}}

\newcommand \wti[1]{{\widetilde {#1}}}

\newcommand{\bC}{\mathbb{C}}
\newcommand{\bF}{\mathbb{F}}
\newcommand{\bT}{\mathbf{T}}
\newcommand{\bM}{{\mathbf{M}}}
\newcommand{\lam}{\lambda}
\DeclareMathOperator{\res}{res}
\newcommand{\la}{\langle}
\newcommand{\ra}{\rangle}
\newcommand{\ess}{{\mathrm{ess}}}
\DeclareMathOperator{\Ell}{Ell}
\DeclareMathOperator{\im}{im}

\newcommand{\Ress}{\mathcal{R}_{\mathrm{ess}}^p}
\newcommand{\Rcpt}{\mathcal{R}_{\mathrm{cpt}}^p}

\newcommand{\cO}{\mathcal{O}}

\newcommand{\cpt}{{\mathrm{cpt}}}

\usepackage{marginnote}

\usepackage[alphabetic,lite]{amsrefs}

\usepackage{color}
	\definecolor{bethorange}{rgb}{0.98, 0.3, 0.0}
	\definecolor{amblue}{rgb}{0.32,0.09,0.98}

\numberwithin{equation}{section}

\begin{document}

\bigskip

\title[A tempered Fourier transform]{A nonabelian Fourier transform for tempered unipotent representations, II}

\author{Anne-Marie Aubert}
\address[A.-M. Aubert]{Sorbonne Universit\'e and Universit\'e Paris Cit\'e, CNRS,
Institut de Math\'ematiques de Jussieu -- Paris Rive Gauche,  
F-75006 Paris, France}
\email{anne-marie.aubert@imj-prg.fr}

\author
{Dan Ciubotaru}
        \address[D. Ciubotaru]{Mathematical Institute, University of Oxford, Oxford OX2 6GG, UK}
        \email{dan.ciubotaru@maths.ox.ac.uk}

\author
{Beth Romano}
        \address[B. Romano]{Department of Mathematics, King's College London, London WC2R 2LS, UK}
        \email{beth.romano@kcl.ac.uk}

\begin{abstract} We extend the construction of the nonabelian Fourier transform proposed in our previous paper from the space of unipotent elliptic characters to the space of unipotent compact characters of a reductive $p$-adic group, and we investigate its relation with the branching to maximal compact open subgroups. We prove that this Fourier transform is compatible with parahoric restriction for certain groups, including $\mathrm{SL}_n, \mathrm{GL}_n$, and $G_2$. We also relate the compact setting to the elliptic setting using parabolic induction.
\end{abstract}

\keywords{nonabelian Fourier transform, unipotent representations, elliptic representations}

\subjclass[2010]{22E50, 20C33}

\maketitle

\setcounter{tocdepth}{1}
\tableofcontents

\section{Introduction}

Unipotent representations of $p$-adic groups, first defined in \cite{LuI}, exhibit the beautiful relations between representations of $p$-adic groups and representations of finite groups of Lie type. To understand them, a natural and often-used strategy is to lift information from these finite groups, whose representation theory is well understood via Deligne--Lusztig theory. In particular, an outstanding question is to give a ``local proof'' of the stability in $L$-packets in the local Langlands correspondence for Lusztig's unipotent representations. This was proven in \cite{MW} for the group $\SO_{2n + 1}$, and one of the main ingredients of the proof was a lift of Lusztig's nonabelian Fourier transform from the setting of finite reductive groups to that of the $p$-adic group $\SO_{2n + 1}$. 

Inspired by this, our previous paper \cite{ACR} (building on \cite{Wa2} and \cite{Ciu}) proposed a lift of Lusztig's nonabelian Fourier transform for the $p$-adic groups that arise as pure inner twists of a split simple group $\bG$. This lift is an involution defined on the space of unipotent elliptic characters, which is isomorphic to the quotient of the Grothendieck group of unipotent representations by the space of parabolically induced representations.

In this paper, we first extend the involution of \cite{ACR} to the essentially elliptic unipotent representation space of reductive groups (inner to split) and then to the space of unipotent compact characters, which is isomorphic to the rigid quotient studied in \cite{CH2}. To do so, we first describe a basis for this space in terms of triples of elements in the dual group $G^\vee$ built out of the (essentially) elliptic pairs, as defined for semisimple groups in \cite{ACR}, for all Levi subgroups $M^\vee$ of $G^\vee$. We then give evidence that our proposed involution is compatible with Lusztig's nonabelian Fourier transform via restriction to maximal compact subgroups. This is done via the induction and restriction functors that allow us to relate the action of the nonabelian Fourier transform on the (essentially) elliptic spaces for the Levi subgroups to the compact space for the group itself.

\subsection{} To give more detail, let $F$ be a nonarchimedean local field with residue field $\mathbb F_q$. Let us assume that $\bG$ is a split connected reductive group over $F$ and $G = \bG(F)$. Let $\InnT^p(G)$ denote the set of equivalence classes of pure inner twists of $G$. The unramified local Langlands correspondence partitions the set of (equivalence classes of) irreducible tempered  unipotent $G'(F)$-representations, $G' \in \InnT^p(G)$, into $L$-packets $\{\pi(su,\phi)\mid \phi\in \widehat {A_{G^\vee}(su)}\}$ labelled by the $G^\vee$-conjugacy classes of elements $x=su \in G^\vee$ (Jordan decomposition) such that $s$ is compact. The elements in the $L$-packet are parametrized by irreducible representations $\phi$ of the group of components $A_{G^\vee}(x)$ of the centralizer of $x$ in $G^\vee$.  Let $\Gamma_u$ denote the reductive part of the centralizer of $u$ in $G^\vee$. In previous works, one considered the set $\cY(\Gamma_u)$ of pairs $(s,h)\in \Gamma_u^2$ of commuting semisimple elements  and the subset $\cY(\Gamma_u)_\ellip$ of elliptic pairs  (see Section \ref{s:ellipticpairs}). 

Each group $G'\in \InnT^p(G)$ has a finite collection of conjugacy classes of maximal compact open subgroups $\max(G')$.  A compact group $K\in \max(G')$ has a finite quotient $\overline K$ that is the group of $\mathbb{F}_q$-points of a (possibly disconnected) reductive group over $\mathbb{F}_q$. Write $R_\un(\overline{K})$ for the $\bC$-vector space spanned by the irreducible unipotent representations of $\overline K$. In \cite{ACR}, we defined an extension of Lusztig's nonabelian Fourier transform \cite{Lubook} to disconnected finite groups in the spirit of \cite{Lu86} and \cite[Section 5]{DM}:
\begin{equation}
\FT_{\cpt,\un}\colon\cC(G)_{\cpt,\un}\to \cC(G)_{\cpt,\un}, 
\end{equation}
on the space 
\begin{equation*}
    \cC(G)_{\cpt,\un}=\bigoplus_{G' \in \InnT^p(G)} \bigoplus_{K \in \max (G')} R_\un(\overline K),
\end{equation*} which we can think of as the sum over $K\in\max(G)$ of the unipotent representation spaces of the pure inner twists of $\overline K$.
In general, $\FT_{\cpt,\un}$ mixes the pure inner twists of a given $\overline K$.

In the case when $\bG$ is semisimple and adjoint, we showed in \cite{ACR} that the local Langlands correspondence induces an isometric isomorphism (for the respective elliptic pairings)
\begin{equation}\label{e:ACR-ell}
\overline{\mathsf{LLC}^p}_{\un}:\bigoplus_{u}\CC[\cY(\Gamma_u)_\ellip]^{\Gamma_u}\longrightarrow \cR^p_{\ellip}(G),\ (s,h)\mapsto \Pi(u,s,h),
\end{equation}
where $\cR^p_{\ellip}(G)=\bigoplus_{G'\in\InnT^p(G)} \overline R_\un(G')$ is the space of elliptic (unipotent) tempered representations for all pure inner twists. The element $u$ ranges over representatives of unipotent conjugacy classes in $G^\vee$ and $\Pi(u,s,h)$ are the endoscopic virtual characters, as defined in (\ref{eqn:pi_ush}). We now show that this isomorphism holds for all split, semisimple $\bG$, thus proving Conjecture 9.1 of \cite{ACR} (see Theorem \ref{t:ess}). 

\medskip

To extend the isomorphism, we introduce the larger space of {\em compact pairs}, $\cpt(G^\vee) := \bigcup_{u \in G^\vee_{\un}} \cY(\Gamma_u)/_\sim$ (see Section \ref{s:ellipticpairs}), with which we replace the left-hand side of (\ref{e:ACR-ell}). The corresponding representation-theoretic space is that of compact representations 
\[
 \Rcpt(G)=\text{span}_\bC\{ \Pi(u,s,h)_c\mid u\in G^\vee_\un/G^\vee,\ (s,h)\in \cY(\Gamma_u)/_\sim\}.
 \]
 where $\Pi(u,s,h)_c$ denotes the character restriction to the space $\cC(G)_{\cpt,\un}$. As shown in Proposition \ref{p:cpt-space}, the space $\Rcpt(G)$ is naturally isomorphic to the rigid quotient space studied in \cite{CH2} from the perspective of the trace Paley--Wiener Theorem, see Section \ref{s:comact-space}. More specifically, we define $\overline{R}(G)_{\un, c}^p$ as in Section \ref{s:comact-space} by taking the quotient of the Grothendieck group of unipotent representations by the $\bC$-span of $i_{M'}^{G'}(\sigma) - i_{M'}^{G'}(\sigma \otimes \chi)$, where $G'$ ranges over $\InnT^p(G)$, $M'$ ranges over the set of standard Levi subgroups of $G'$, $\sigma \in R_{\un}(M')$, and $\chi$ is an unramified character. By studying the compatibility of the constructions with respect to parabolic induction, we prove:
 \begin{thm}[Proposition \ref{p:cpt-space}]\label{t:cpt-intro}
    The map $\overline{\mathsf{LLC}^p}_{\un}$ can be lifted to an isomorphism
     \[
     \overline{R}(G)_{\un, c}^p \cong \Rcpt(G)\cong \bigoplus_{u \in G^\vee_{\un}/G^\vee} \bC[\cY(\Gamma_u)/_\sim]^{\Gamma_u}.
     \]
 \end{thm}
We mention that for this proof, we need to first extend (\ref{e:ACR-ell}) to cover the more general case of reductive groups, in which case the space of elliptic pairs is enlarged to a space of {\it essentially elliptic} pairs. See Section \ref{s:ess-ell}.

\begin{rem}
    The isomorphisms in (\ref{e:ACR-ell}) and Theorem \ref{t:cpt-intro} indicate that there exists a deeper, geometric connection between the two sides. Indeed, in \cite[Theorem E and Remark 1.3.6]{BKK}, Bezrukavnikov, Karpov, and Krylov have shown, via a categorification of Lusztig's asymptotic Hecke algebra, that the cocenter of the affine Hecke algebra with equal parameters is naturally isomorphic to the sum $\bigoplus_u \mathcal O^a(\Gamma_u)$, where $\mathcal O^a(\mathcal C_{\Gamma_u})$ is the ring of regular functions on the commuting variety $\mathcal C_{\Gamma_u}=\text{Comm}_{\Gamma_u}//\Gamma_u$ that are left-locally constant and ``of Springer type'' (see the precise conditions (a), (b) in {\it loc. cit.}). 

    More generally, for the full unipotent part of the cocenter of a split reductive group, a similar isomorphism with the rings of left-locally constant functions on $\mathcal C_{\Gamma_u}$ is constructed by a different geometric realization in upcoming work by Bezrukavnikov, Kazhdan, Varshavsky, and the second author.

    We mention that related categorifications of the cocenter and the relation with functions on commuting varieties in the Langlands dual group have been obtained by Ben-Zvi--Nadler--Preygel \cite{BNP} and Li--Nadler--Yun \cite{LNY}.
\end{rem}

\subsection{}
Since $\Rcpt(G)$ has an obvious involution given by the flip $(s,h)\to (h,s)$, this defines, as in the elliptic case, an involution, the dual nonabelian transform 
\begin{equation}
\FT^\vee_\cpt\colon \Rcpt(G) \to \Rcpt(G).
\end{equation}
We remark that $\FT^\vee_\cpt$ mixes the representations of the pure inner twists of $G$. As in the elliptic case of \cite{ACR}, we expect that there is a commutative diagram (up to certain roots of unity):
 \begin{displaymath}
 \xymatrix@+1pc{
    {\Rcpt(G)} \ar[r]^{\mathrm{FT}_\cpt^\vee} \ar[d]_{\res_{\cpt,\un}}
    & {\Rcpt(G)}\ar[d]^{\res_{\cpt,\un}}\\
    {  \cC(G)_{\cpt,\un}} \ar[r]_{\mathrm{FT_{\cpt,\un}}}
    & {  \cC(G)_{\cpt,\un}} }
\end{displaymath} 
where the vertical arrows are defined by taking invariants by the pro-unipotent radicals of maximal compact subgroups. More specifically, we expect the following:

\begin{conj}[Conjecture \ref{c:compact}]
For each unipotent element $u \in G^\vee$, compact pair $(s, h) \in \cY(\Gamma_u)/_\sim$, and maximal compact open subgroup $K$ of $G$, there exists a root of unity $\zeta$ such that
\begin{equation*}
\res_K(\Pi(u, h, s)_c) = \zeta (\FT_K \circ \res_K)(\Pi(u, s, h)_c).
\end{equation*}
\end{conj}
Here $\res_K$ is parahoric restriction with respect to the compact subgroup $K$, which has image in $\oplus_{H \in \InnT(\overline K)} R_\un(H)$, and $\FT_K$ is the nonabelian Fourier transform on this space, as extended to disconnected groups in \cite{ACR}.
We give evidence for this conjecture, including the following:
\begin{thm}[Theorems \ref{t:SLn} and \ref{t:PGL}]\label{t:typeA-intro} Conjecture \ref{c:compact} holds when $\bG = \SL_n$ and when $\bG = \PGL_m$ for $m$ prime.
\end{thm}

Parabolic induction gives a map from the essentially elliptic spaces for Levi subgroups of $G$ to the space $\Rcpt(G)$. Using this map, we show that, under some natural hypotheses (Hypotheses \ref{h:FT-comm} and \ref{h:FT-twist}) about Lusztig's Fourier transform and the local Langlands correspondence, we prove the following:

\begin{thm}[Corollary \ref{c:levis-imply-compact}]
Assume $G$ has simply connected derived subgroup. Then Conjecture \ref{c:compact} holds for $G$ if the corresponding conjecture on elliptic spaces holds for all Levi subgroups $M$ of $G$.
\end{thm}
See Section \ref{s:FT-ind} for more details and a more precise statement.
Using this result, we prove the following corollary (again, assuming Hypotheses \ref{h:FT-comm} and \ref{h:FT-twist}):
\begin{cor}[Corollary \ref{c:GL-G2}]
Conjecture \ref{c:compact} holds when $\bG$ is $\GL_n$ or $G_2$.
\end{cor}

\subsection{Structure of the paper}
Section \ref{s:prelim} contains background and preliminary results about Levi subgroups and parabolic induction for $p$-adic groups. In Section \ref{s:ellipticpairs}, we define compact pairs in complex reductive groups. We also extend the notion of elliptic pairs from \cite{ACR} to define essential elliptic pairs. We give an explicit relation between compact pairs for a group $G^\vee$ and essentially elliptic pairs in Levi subgroups $M^\vee$ of $G^\vee$. 

In Section \ref{s:ess-ell} we define the essentially elliptic space $\Ress(G)$ and extend the results of \cite{ACR} to this setting, including a proof of (a generalization of) \cite[Conjecture 9.1]{ACR}. In Section \ref{s:rigid} we define the compact space $\Rcpt(G)$, prove Theorem \ref{t:cpt-intro}, and formulate Conjecture \ref{c:compact}. In Section \ref{s:FT-ind} we prove the relation between the compact and essentially elliptic conjectures. Finally, in Section \ref{s:ex-A} we consider the case when $\bG$ is $\SL_n$ or $\PGL_n$, proving Theorem \ref{t:typeA-intro}.

\subsection{Notation and conventions}\label{section-notation}

Given a (possibly disconnected) complex reductive group $\cG$, we write $\rZ_{\cG}$ for the center of $\cG$. Given $x \in \cG$, we write $\rZ_{\cG}(x)$ for the centralizer of $x$ in $\cG$. Similarly, if $H$ is a subgroup of $\cG$, we write $\rZ_{\cG}(H)$ for the centralizer of $H$ in $\cG$, and if $\varphi$ is a homomorphism with image in $\cG$, we write $\rZ_{\cG}(\varphi)$ for $\rZ_{\cG}(\text{im } \varphi)$. We write $\cG^\circ$ for the identity component of $\cG$. If $x \in \cG$, we write $A_{\cG}(x) = \rZ_{\cG}(x)/\rZ_{\cG}(x)^\circ$ for the component group of $\rZ_{\cG}(x)$. We write $\cG_\un$ for the set of unipotent elements of $\cG$. If $u \in \cG_\un$, we write $\Gamma^{\cG}_u$ for the reductive part of $\rZ_{\cG}(u)$ (if there is no risk of confusion, we simply write $\Gamma_u$).
We say that a subgroup $L \subset \cG$ is a Levi subgroup if $L$ is the centralizer of a torus in $\cG$. 

If $A$ is a group, $H$ is a subgroup of $A$, $\rho$ is a representation of $H$, and $g \in A$, we write $\rho^g$ for the twisted representation of $H^g:= gHg^{-1}$. If a group $A$ acts on a set $S$, we write $S/A$ for the set of $A$-orbits on $S$. We sometimes write $A\setminus S$ if $A$ is acting on the left.
Given a finite group $A$, we write $\widehat{A}$ for the set of irreducible characters of $A$.
Given a finite set $S$, we write $\mathbb{C}[S]$ for the $\mathbb{C}$-vector space of functions $S \to \mathbb{C}$.

Given a reductive group $\bH$ defined over a finite or nonarchimedean local field $k$, we use a Roman letter $H$ for $\bH(k)$. We write $\Irr_{\un}(H)$ for the set of (isomorphism classes of) irreducible unipotent representations of $H$ and $R_{\un}(H)$ for the $\bC$-span of $\Irr_{\un}(H)$. Given a parabolic subgroup $P = MN$ of $H$, we write $i_P^H: R_{\un}(M) \to R_{\un}(H)$ for the linear map given by parabolic induction on irreducible representations. We also denote this map by $i_M^H$. We write $r_H^M: R_{\un}(H) \to R_{\un}(M)$ for Jacquet restriction.

Throughout the paper, $F$ is a nonarchimedean local field, $\bG$ is a split connected reductive group over $F$, and $G = \bG(F)$.

\medskip

\noindent{\bf Acknowledgments.} We are grateful to Jean-Loup Waldspurger and Tom Haines for valuable emails. D.C. thanks Roman Bezrukavnikov for many helpful discussions about the geometrization of the cocenter and characters of $p$-adic groups.

\section{Preliminaries}\label{s:prelim}

We now review some background and prove some preliminary results. We broadly follow the conventions of our previous paper. See \cite[Section 4]{ACR} for more details. Let $F$ be a nonarchimedean local field, let $F_s$ be a fixed separable closure of $F$, and let $F_{\un} \subset F_s$ be the maximal unramified extension. Let $\mathfrak{o}_F$ denote the ring of integers of $F$. Let $W_F$ be the Weil group of $F$, and let $\Frob \in W_F$ be the geometric Frobenius element of $\Gal(F_{\un}/F)$. 

Let $\bG$ be a split connected reductive group over $F$. Let $G = \bG(F)$. As in \cite[Section 2]{ACR}, we write $\InnT^p(G)$ for the set of pure inner twists of $G$. There is natural bijection \[\InnT^p(G)\longleftrightarrow H^1(F, \bG),\]
where $H^1(F,\bG)$ denotes the first Galois cohomology group of $\bG$.

\subsection{The Unipotent Local Langlands Correspondence}\label{s:llc}
We write $\Irr_{\un}(G)$ for the set of isomorphism classes of irreducible unipotent representations of $G$, and write $R_{\un}(G)$ for the $\bC$-span of the elements of $\Irr_{\un}(G)$.

Let $G^\vee$ be (the $\bC$-points of) the dual group of $\bG$. 
Since $\bG$ is split, we may think of Langlands parameters as continuous morphisms $\varphi: W_F \times \SL_2(\bC) \to G^\vee$ such that $\varphi(w)$ is semisimple for all $w \in W_F$ and the restriction of $\varphi$ to $\SL_2(\bC)$ is algebraic. Two Langlands parameters are considered to be equivalent if they are conjugate by an element of $G^\vee$. 
A Langlands parameter $\varphi$ is called unramified if the restriction of $\varphi$ to the inertia subgroup of $W_F$ is trivial. 
Write $\Phi_{\un}(G)$ for the set of equivalence classes of unramified Langlands parameters for $G$.
Recall that an unramified Langlands parameter $\varphi$ is completely determined up to conjugacy by the pair $(u_\varphi, s_\varphi)$, where 
\begin{equation*}
u_\varphi := \varphi(1, \begin{psmallmatrix} 1 & 1\\0 & 1\end{psmallmatrix}) \text{ and } s_\varphi := \varphi(\Frob,1)
\end{equation*}
Conversely, given $x \in G^\vee$ with Jordan decomposition $x = us$, there exists an unramified Langlands parameter $\varphi$, which is unique up to $G^\vee$-conjugacy, such that $u = u_\varphi$ and $s = s_\varphi$. Write $\varphi_{(u, s)}$ for the element of $\Phi_{\un}(G)$ corresponding to $(u, s)$ in this way. 

Given an unramified Langlands parameter $\varphi$, let $A_\varphi = \rZ_{G^\vee}(\text{im } \varphi)/\rZ_{G^\vee}(\text{im } \varphi)^\circ$. Let 
\begin{equation*}
\Phi^p_{e, \un}(G^\vee) = G^\vee \setminus \{ (\varphi, \rho) \mid \varphi \text{ is an unramified Langlands parameter, } \rho \in \widehat{A}_\varphi\}.
\end{equation*}
 The unramified local Langlands correspondence (first proven for simple adjoint $\bG$ in \cite{LuI} and proven in the most general case in \cite{So1}) determines a bijection
 \begin{equation*}
\bigsqcup_{G' \in \InnT^p(G)} \Irr_{\un}(G') \longrightarrow \Phi^p_{e, \un}(G^\vee).
 \end{equation*}
 Write $\pi(u, s, \rho)$ for the preimage of $(\varphi_{(u, s)}, \rho)$ under this bijection. For clarity, if $\pi(u, s, \rho)$ is representation of $G' \in \InnT^p(G)$, we sometimes write $\pi(u, s, \rho)$ as $\pi^{G'}(u, s, \rho)$.

Given an enhanced Langlands parameter $(\varphi_{(u, s)}, \rho)$, let $G_\rho \in \InnT^p(G)$ be the pure inner twist such that $\pi(u, s, \rho)$ is a representation of $G_\rho$ (this is the pure inner twist $G'$ such that the representation of $\pi_0(Z_{G^\vee})$ induced by $\rho$ corresponds to $G'$ under the identification $\InnT^p(G) \to H^1(k, \bG) \to \Irr(\pi_0(Z_{G^\vee}))$. See \cite[Section 2.1]{ACR}, for example. 

The following lemma will be useful later. In it, we use $\Gamma_u$ for $\rZ_{G^\vee}(u)^{\red}$ (cf. Section \ref{section-notation}).

\begin{lem}\label{l:center-twist}
Let $(\varphi_{(u, s)}, \rho) \in \Phi_{e, \un}^p(G^\vee)$. 
Suppose $s' \in \Gamma_u$ is a semisimple element such that $\gamma s \gamma^{-1} = s'z$ for some $\gamma \in \Gamma_u$ and some $z \in \rZ_{G^\vee}^\circ$. Then there exists an unramified character $\chi$ of $G_\rho$ such that $\pi(u, s, \rho) = \pi(u, s', \rho^\gamma) \otimes \chi$. In particular, for any compact subgroup $K$ of $G_\rho$, we have $\pi(u, s, \rho) \vert_K = \pi(u, s', \rho^\gamma) \vert_K$.
\end{lem}

\begin{proof}
We have
\begin{equation*}
\pi(u, s, \rho) = \pi(u, \gamma s \gamma^{-1}, \rho^{\gamma}) = \pi(u, s'z, \rho^\gamma) = \pi(1, s, \rho^\gamma) \otimes \chi_z,
\end{equation*}
where $\chi_z$ is the unramified character of $G_\rho$ corresponding to $z$ as in \cite[Section 3.3.1]{Hai14}. 
\end{proof}

\subsection{Levi subgroups}

Let $M^\vee \subset G^\vee$ be a Levi subgroup. 
Since $G$ is split, $M^\vee$ is relevant for $G$ as defined in \cite[Section 3.4]{Bor}, so $M^\vee$ determines conjugacy class of Levi subgroups of $G$ via the bijection of \cite[Section 3.3]{Bor}. 

Next we recall how pure inner twists of Levi subgroups of $G$ give rise to pure inner twist of $G$, as explained in \cite[Section 5]{Pra1}. Let $\mathbf{P}$ be a parabolic subgroup of $\bG$ with Levi decomposition $\mathbf{P} = \bM\mathbf{N}$. By \cite[Theorem 4.13]{BT1}, the natural map $H^1(F, \mathbf{P}) \to H^1(F, \bG)$ is injective. We have $H^j(F, \mathbf{N}) = 0$ for all $j$, so the short exact sequence
\begin{equation*}
1 \to \mathbf{N} \to \mathbf{P} \to \bM \to 1
\end{equation*}
determines an isomorphism $H^1(F, \bM) \to H^1(F, \mathbf{P})$. Thus we have a natural injective map 
\begin{equation}\label{eqn-H1-Levi}
\iota_{\mathbf{M},\mathbf{G}}\colon H^1(F, \mathbf{M}) \hookrightarrow H^1(F, \bG).
\end{equation}

Thus (\ref{eqn-H1-Levi}) determines a natural injection
\begin{equation}\label{eqn-Inn-Levi}
\InnT^p(M) \hookrightarrow \InnT^p(G),
\end{equation}
where $M = \bM(F)$. 
In this way, each pure inner twist of a Levi $M$ of $G$ determines a unique pure inner twist of $G$.

Let $z\in H^1(\bG,F)$ represent $G'$ (a pure inner twist of $G$). A given Levi subgroup $M$ of $G$ admits a pure inner twist $M'$ that is a Levi subgroup of $G'$ if and only if $z$ lies in the image of $\iota_{\mathbf{M},\mathbf{G}}$. In other words, there is a bijection between the set of $G'$-conjugacy classes of Levi subgroups of $G'$ and the disjoint union over the $G$-conjugacy classes of Levi subgroups $M$ of $G$ such that $\iota_{\mathbf{M},\mathbf{G}}(z^\bM)=z$ with $z^\bM\in H^1(\bM,F)$. If a Levi subgroup $M$ of $G$ has this property, we will call it \emph{$G'$-relevant}. Not all Levi subgroups $M$ of $G$ are $G'$-relevant.
A Levi subgroup $M$ of $G$ is $G'$-relevant if and only if there exists a representative $c\colon\sigma\mapsto c_\sigma$ of $z$ in $Z^1(\bG,F)$ that satisfies $c_\sigma\in\Nor_\bG(\bM)$ for every $\sigma\in\Gal(\overline{F}/F)$.

\begin{ex} Take $G=\PGL_n(F)$ and $G'=\PGL_m(D)$ where $D$ is a central division over $F$ of degree $d$ with $n=md$. Then $G'$ is a pure inner twist of $G$. A Levi subgroup $M$ of $G$ is the image under the canonical projection $\GL_n(F)\to\PGL_n(F)$ of a Levi subgroup $\GL_{n_1}(F)\times\cdots\times\GL_{n_r}(F)$ of $\GL_n(F)$. Then  $M$ is a pure inner twist of a Levi subgroup $M'$ of $G'$ if and only if each $n_i$ is divisible by $d$. If so, let $m_i:=n_i/d$, then the Levi $M'$ is the image of $\GL_{m_1}(D)\times\cdots\times\GL_{m_r}(D)$ under the canonical projection $\GL_m(D)\to\PGL_m(D)$.
\end{ex}

\begin{lem} \label{lem:Levis}
Let $G'$ be a pure inner twist of $G$, and let $M'$ be a Levi subgroup of $G'_{F_s}$. Then $M'$ is a pure inner twist of a Levi subgroup $M$ of $G$.
\end{lem} 
\begin{proof} Fix a split maximal torus $T\subset G$ and a Borel subgroup $B\supset T$. There exists $\xi\colon G_{F_s}\to G_{F_s}'$ satisfying $\xi^{-1}\circ \sigma(\xi)=\Ad(z_\sigma)$ for all $\sigma\in\Gal(F_s/F)$, with $z\in Z^1(F,G)$, and $z_\sigma:=z(\sigma)$. Let $M'$ be a Levi subgroup of $G'_{F_s}$, and set $\tilde M:=\xi^{-1}(M')$. Then $\tilde M$ is a Levi subgroup of $G_{F_s}$. Hence, there exists $g\in G(F_s)$ such that $g\tilde Mg=M$ for some standard Levi subgroup $M$ of $G$ (that is, such that $M\supset T$). By applying $\xi$, we obtain $\xi(g) M'\xi(g)^{-1}=\xi(M)$. So, setting $\tilde\xi:=\xi\circ \Ad(g^{-1})$, we get  $\tilde\xi(M_{F_s})=M'_{F_s}$. Then, for every $\sigma\in\Gal((F_s/F)$, we have 
$\tilde\xi^{-1}\circ \sigma(\tilde\xi)=\Ad(\tilde z_\sigma)$, where $\tilde z_\sigma:=gz_\sigma\sigma(g)^{-1}$. It gives
\[\Ad(\tilde z_\sigma)(M)=\tilde\xi^{-1}(\sigma(\tilde\xi(M))=\tilde\xi^{-1}(\sigma(M'))=\tilde\xi^{-1}(M')=M,\]
that is, $\Ad(\tilde z_\sigma)\in\Nor_G(M)(F_s)$. So the cocycle $z$ actually defines a class $[\tilde z]\in H^1(F,\Nor_G(M))$.
We define a twisted Galois action on $M_{F_s}$ by \[\sigma\ast m:=\tilde z_\sigma\sigma(m)\tilde z_\sigma^{-1}, \quad\text{for every $\sigma\in\Gal(F_s/F)$.}\]
Let $w_\sigma$ denote the class of $\tilde z_\sigma$ in $W(M):=\Nor_G(M)/M$. The map $w\colon\sigma\mapsto w_\sigma$ is a $1$-cocycle, that is, for all $\sigma,\tau\in\Gal(F_s/F)$, we have $w_{\sigma\tau}=w_\sigma\cdot {}^{\sigma} w_\tau=w_\sigma\cdot w_\tau$, since $\Gal(F_s/F)$ acts trivially on $W(M)$ because $G$ is split. We will see that $w$ is actually a $1$-coboundary.
The fact that $\Ad(\tilde z_\sigma)\in\Nor_G(M)(F_s)$ shows that $\sigma\ast M_{F_s}=M_{F_s}$. Hence $\Ad(\tilde z_\sigma)|_{M}$ is an automorphism of $M$, and since $M$ is $F$-split, it is an inner automorphism of $m$. Thus, there exists $m_\sigma\in M$ such that $\Ad(\tilde z_\sigma)|_{M}=\Ad(m_\sigma)$, that is, since $\Gal(F_s/F)$ acts trivially on $M$, 
\[\tilde z_\sigma m\tilde z_\sigma^{-1}=m_\sigma m m_\sigma^{-1}.\] 
Hence, we have $m_\sigma^{-1}\tilde z_\sigma\in\Cent_G(M)\subset M$.
It follows that $w_\sigma=1$ in $H^1((F,W(M))$. Hence, there exists $w\in W(M)$ such that $w_\sigma=w^{-1}\sigma(w)$. Pick a lift $n\in\Nor_G(M)(F_s)$ of $w$. We set $z'_\sigma:=n\tilde z_\sigma\sigma(n)^{-1}$ and $\xi':=\tilde\xi\circ\Ad(n)$.  We have $z'\in Z^1(F,M)$ and $\xi'\circ\sigma(\xi')=\Ad(z'_\sigma)$. The restriction $\xi'|_{M}\colon M_{F_s}\to M'_{F_s}$ satisfies 
$\xi'|_{M}\circ\sigma(\xi'|_{M})=\Ad(z'_\sigma)$ on $M_{F_s}$, that is, $M'$ a pure inner twist of $M$. The above construction required choices (choosing a conjugating $g$ and a lift $n\in \Nor_G(M)$ of $w\in W(M)$. A different choice replaces $z'$ by a cohomologous cocycle, so the class in $H^1(F,M)$ is determined up to the natural equivalence.
\end{proof}

\subsection{Parahoric restriction and parabolic induction}\label{s:ind-res-prelim}

We need to know how the operations of parahoric restriction and parabolic induction interact.  Let $K_0$ be a parahoric subgroup of $G$ and $K_0\le K\le N_G(K_0)$. Let $U_K$ denote the pro-unipotent radical of $K$, and write $\overline{K} = K/U_K$ for its reductive quotient. The parahoric restriction of an admissible $G$-representation $V$ is $\res_K^G(V) = V^{U_K}$, viewed as a finite-dimensional $K$- (or $\overline K$-) representation. 
  We also write $\res_K^G: R_{\un}(G) \to R_{\un}(\overline{K})$ for the corresponding linear map on Grothendieck groups. We may write $\res_K$ if the group $G$ is understood.
Given a parabolic subgroup $P$ of $G$, write $\overline{P \cap K}$ for the quotient $(P \cap K)/(P \cap U_K)$, and similarly for Levi subgroups of $G$. 

Fix a maximal split torus $T \subset G$, and let $\cA_T$ be the apartment associated to $T$ in the Bruhat--Tits building of $G$. Fix an alcove $\mathfrak{a}$ in $\cA_T$ and a hyperspecial vertex $x_0$ in the closure of $\mathfrak{a}$. These choices determine a Borel subgroup $B$ containing $T$ (cf. \cite[Section 2.5]{Hai09}), and we call a parabolic subgroup $P$ of $G$ standard if $B \subset P$. We call a Levi subgroup $M$ of $G$ standard if it is the Levi factor of a standard parabolic. The choice of basepoint $x_0$ allows us to identify $\cA_T$ with $X_*(T) \otimes \bR$ (here $X_*(T)$ is the cocharacter lattice of $T$). Let $W = W_G$ be the Weyl group of $G$ with respect to $T$.

Given $g \in G$ and a subgroup $H$ of $G$, we write $H^g$ for $gHg^{-1}$. We first prove two preliminary lemmas about the structure of $G$.

\begin{lem}\label{l:levi-compact}
Let $n \in N_G(T)$, and let $K = K_x$ be the parahoric subgroup associated to a point $x \in \cA_T$. Let $P = MN$ be a standard parabolic subgroup of $G$. Then 
\begin{itemize}
\item[(1)] $M^n \cap K$ is a parahoric subgroup of $M$, and 
\item[(2)]  $\overline{P^n \cap K}$ is a parabolic subgroup of $\overline K$ with Levi decomposition $\overline{P^n \cap K} = (\overline{M^n \cap K})(\overline{N^n \cap K})$. 
\end{itemize}
\end{lem}

\begin{proof}
Thinking of $\bG$ as a Chevalley group defined over $\bZ$ with maximal torus $\bT$ and corresponding root system $\Phi$, we fix root-group homomorphisms $u_\al: \mathbb{G}_a \to \bG$ for the roots $\al \in \Phi$ compatible with our choice of hyperspecial vertex $x_0$ above. 
Let $\Phi_M$ be the sub-root system corresponding to the Levi $M$. Then $P$ is generated by $T$ and by the root groups $U_\al$ such that $\al \in \Phi_M$ or $\al > 0$. Write $w$ for the image of $n$ in the Weyl group, let $\cO$ denote the ring of integers of $F$, and let $\varpi \in \cO$ be a uniformizer. We have that
\begin{equation*}
M^n \cap K = \la \bT(\cO), u_{w(\al)}(\varpi^m\cO) \mid \al \in \Phi_M, w\al(x) + m \geq 0 \ra.
\end{equation*}
If we identify $\cA_T$ with the apartment of $T$ in the Bruhat--Tits building of $M$, it is clear that this is the parahoric subgroup of $M$ corresponding to $x$. This proves (1).

Next, we have
\begin{equation*}
P^n \cap K = \la \bT(\cO), u_{w(\al)}(\varpi^m\cO) \mid \al \in \Phi_M \cup \Phi^+, w\al(x) + m \geq 0 \ra,
\end{equation*}
and similarly with $P^n \cap U_K$ (replacing the inequality with a strict inequality and $\bT(\cO)$ with the appropriate subgroup of $T$). 
Given a point $y \in \cA_T$, let $\Phi_y = \{\al \in \Phi \mid \al(y) \in \bZ\}$. We identify $\overline K$ with the subgroup $\la \bT(\Fq), u_{w(\al)}(\Fq) \mid \al \in \Phi_{w^{-1}(x)} \ra$ of $\bG(\Fq)$. Under this identification, we have
\begin{equation*}
\overline{P^n \cap K} = \la \bT(\Fq), u_{w(\al)}(\Fq) \mid \al \in (\Phi_M \cup \Phi^+) \cap \Phi_{w^{-1}(x)} \ra,
\end{equation*}
a subgroup of $\overline K$. Let $Q = \langle \bT(\Fq), u_{w(\al)}(\Fq) \mid \al \in \Phi_M \cup \Phi^+\rangle$, a parabolic subgroup of $\bG(\Fq)$. Then $Q \cap \overline K = \overline{P^n \cap K}$ is a parabolic subgroup of $\overline K$  by \cite[Proposition 2.2]{DM91}. 

The fact that $\overline{M^n \cap K}$ is a Levi subgroup of $\overline{P^n \cap K}$ follows from \cite[Proposition 2.2]{DM91} using similar logic, and the rest of the result follows easily.

\end{proof}

\begin{rem}
We will also need similar results in the case when $K$ is not necessarily a parahoric, i.e. when we have a containment $K_0 \subset K \subset N_G(K_0)$ as above but $K \neq K_0$. In this case the reductive quotient $\overline{K}$ is disconnected with identity component $\overline{K_0}$. 
More specifically, $\overline{K} = \overline{K_0} \rtimes A$ for some finite abelian group $A$ (which is a subgroup of the stabilizer in $G$ of $\mathfrak{a}$). With $P = MN$ as in the lemma, we have that $\overline{P^n \cap K}$ is a parabolic subgroup of $\overline{K}$ with quasi-Levi factor $\overline{M^n \cap K}$ (cf. \cite[4.8.1]{GM}). 
\end{rem}

For every $w\in W$, let $\dot w$ denote a representative in $N_G(T)$. Let $I$ be the Iwahori subgroup corresponding to our fixed alcove $\mathfrak{a} \subset \cA_T$. Given a standard Levi subgroup $M$, write $W_M$ for the parabolic subgroup of $W$ corresponding to $M$. The following lemma is a known consequence of the Iwasawa decomposition, see \cite[\S11.2]{GHKR}.
\begin{lem}\label{l:decomp}
\begin{enumerate}
    \item   $G=\bigsqcup_{w\in W} I \dot w B$.
  \item More generally, if $P = MN$ is a standard parabolic, then $G=\bigsqcup_{w\in W/W_M} I \dot w P$.
    \end{enumerate}
\end{lem}

\begin{proof}
Let $K_0$ be a maximal hyperspecial subgroup containing the Iwahori subgroup $I$, such that $K_0=IWI$. Since $G=K_0B$, we have $G=IWIB$. Use the Iwahori decomposition $I=I^-\cdot (B\cap I)$, and get $G=IWI^-B$. Finally, for every $w\in W$, $w(I^-)\subset I$, so we may move $I^-$ to the left and absorb it into $I$. Hence $G=IWB$. Suppose $\dot w'\in I\dot w B$, then $\dot w'=i\dot w b$, for some $i\in I$, $b\in B$. This implies $b\in B\cap K_0\subset I$, and so $I \dot w' I=I w I$, hence $\dot w'=\dot w$, and (1) is proven. Claim (2) follows from (1).

\end{proof}

Now we consider parabolic induction. 
Given an $F$-parabolic $P = MN$ of $G$, we write $i_P^G: R_\un(M) \to R_\un(G)$ for the linear map given by (normalized) parabolic induction on irreducible representations (we also write $i_M^G$ for this map). We extend $i_P^G$ to a
linear map
\begin{equation}\label{r:par-ind-twists}
\bigoplus_{M' \in \InnT^p(M)} R_{\un}(M') \longrightarrow \bigoplus_{G' \in \InnT^p(G)} R_{\un}(G')
\end{equation}
defined as follows: given $M' \in \InnT^p(M)$ and $\pi \in \InnT^p(M)$, let $G'$ be the image of $M'$ under (\ref{eqn-Inn-Levi}). The map is then given by $\bigoplus_{M' \in \InnT^p(M)} i_{M'}^{G'}$. We continue to write $i_P^G$ (or $i_M^G$) for this linear map.

Now let $P=MN$ be a standard parabolic subgroup. Let $V=i_P^G(V_0)$ be a (normalized) parabolically induced representation, where $(\pi_0,V_0)$ is an irreducible $M$-representation inflated to $P$. By Mackey's formula (see for example \cite[Theorem]{Ku}):
\[
    V|_K=i_P^G(\pi_0)|_K=\bigoplus_{g\in K\backslash G/P} \Ind_{K\cap P^g}^K (\pi_0^g),
\]
where $P^g=gPg^{-1}$ and $\pi_0^g(x)=\pi_0(g^{-1}xg).$ Take $U_K$-fixed vectors (recall that $U_K$ is normal in $K$):
\begin{equation}\label{e:kutzko}
V^{U_K}|_K=\bigoplus_{g\in K\backslash G/P} \left(\Ind_{K\cap P^g}^K (\pi_0^g)\right)^{U_K}
\end{equation}

\begin{rem}
 Let $K$ be a compact subgroup of $G$ containing $I$, and 
    let $P = MN$ be a standard Levi. 
    Note that by Lemma \ref{l:decomp} we can take double-coset representatives for $K\backslash G/P$ in $N_G(T)$.
    Denote by $^K\!W^P\subset N_G(T)$ a set of representatives for the double cosets $K\backslash G/P$. 
\end{rem}

We have the following more detailed description of $^K\!W^P$.

\begin{prop}
    With the same notation as above, there is a one-to-one correspondence between $^K\!W^P$ (and hence $K\backslash G/P$) and $W_K\backslash W/W_M$, where $W_K$ is the image of the projection of the Iwahori--Weyl group of $K$ to the finite Weyl group.
\end{prop}
\begin{proof}
    When $K$ is a parahoric, this is in \cite[Lemma 4.5.2]{Hai09}. The case when $K$ is a compact subgroup containing $I$ is proved in the same way with Lemma \ref{l:decomp} as the starting point. 
\end{proof}

Finally, we give a formula for the parahoric restriction of a parabolically induced representation. In the case when the parahoric is maximal hyperspecial, this is well known \cite[Proposition 3.1.1]{Cas}.

\begin{prop}\label{p:par-ind}
    Let $K$ be a compact subgroup of $G$ containing $I$, and 
    let $P = MN$ be a standard parabolic subgroup. Let $V=i_P^G(\pi_0)$ be a parabolically induced representation. Then 
    \[
    \res_K^G(V)=\bigoplus_{g\in ^K\!W^P} i_{\overline {K\cap P^g}}^{\overline K}(\res^{M^g}_{K\cap M^g}(\pi_0^g)),
    \]
   In particular, if $K$ is maximal hyperspecial, $\res_K(V)=i_{\overline{K\cap P}}^{\overline K}(\res^M_{K\cap M}(V_0)).$
\end{prop}

Note that by Lemma \ref{l:levi-compact} $K \cap M^g$ is a compact subgroup of $M^g$ and $\overline{K \cap P^g}$ is a parabolic subgroup of $\overline{K}$, so the notation in the proposition makes sense. 

\begin{proof}
    In light of (\ref{e:kutzko}), we only need to prove that
    \[
    \left(\Ind_{K\cap P}^K(\pi_0)\right)^{U_K}\cong i_{\overline {K\cap P}}^{\overline K} (\pi_0^{U_{K\cap M}}), \text{ as $\overline K$-representations.}
    \]
   Recall that $U_K$ has an Iwahori decomposition $U_K=(U_K\cap N^-)(U_K\cap M)(U_K\cap N)$, where $P^-=MN^-$ is the opposite parabolic. We use the following conventions for induction:
\[
\Ind_{K\cap P}^K(\pi_0)=\{f:K\to V_0 \text{ loc. const.}\mid f(kmn)=\pi_0(m^{-1})f(k),\ k\in K,\ mn\in K\cap P\},
\]
with the left regular action of $K$: $(k'\cdot f)(k)=f({k'}^{-1}k)$, $k,k'\in K$. 
   This means that if $f:K\to V_0$ is an element in the left hand side, then $f(uk)=f(k)$ for all $u\in U_K$, $k\in K$. This is the same as a function $\bar f:\overline K\to V_0$. Since $U_K$ is normal in $K$, this is also equivalent to $f(ku)=f(k)$, $k\in K, u\in U_K$.  This gives us the fact that $f$ is a function on $\overline K$. Take $u\in U_K\cap P$, then $f(k)=f(ku)=\pi_0(u^{-1}) f(k)$,  where the second equality is by the definition of the induced representation, so $f(k)\in V_0^{U_K\cap P}$. Since $U_K\cap N$ acts trivially on $V_0$, this is the same as $f(k)\in V_0^{U_K\cap M}$, i.e., $f$ determines a function $\bar f:\overline K\to V_0^{U_{K\cap M}}$. (Here we are using that $U_{K\cap M}=U_K\cap M$.) Similarly, it is easy to check that the right hand side is contained in the left hand side. Clearly, this identification is $\overline K$-equivariant. This proves the first claim. 

   If $K$ is maximal hyperspecial, then $G=KP$, hence the second claim.
\end{proof}

\subsection{Aside: reduction to Hecke algebras}
It is useful to give an explicit way of computing the parahoric restrictions of parabolically induced unipotent representations. Fix $J$ a parahoric subgroup of $G$ containing the fixed Iwahori subgroup $I$. Let $(\tau,V_\tau)$ be the $J$-inflation of an irreducible cuspidal unipotent representation of $\overline J$. Let $\mathcal C(G,J,\tau)$ denote the Bernstein component in the category of unipotent representations $V$ of $G$ such that $V$ is generated by $V^\tau=\Hom_J(\tau,V)$. Let $\cH(G,J,\tau)$ denote the convolution Hecke algebra of compactly-supported locally-constant functions
\[
\cH(G,J,\tau)=\{f:G\to\operatorname{End}(V_\tau^*)\mid f(j_1gj_2)=\tau^*(j_1)\circ \tau(g)\circ\tau^*(j_2),\ j_1,j_2\in J,\ g\in G\},
\]
where $(\tau^*,V_\tau^*)$ denotes the contragredient representation. There is an equivalence of categories functor
\[
m_\tau: \mathcal C(G,J,\tau)\to \cH(G,J,\tau)\text{-mod},\quad V\mapsto V^\tau.
\]
In the language of \cite{BK}, the pair $(J,\tau)$ is a $G$-type.

Let $(L,\sigma_L)$ denote the supercuspidal support of $\mathcal C(G,J,\tau)$: $L$ is a standard Levi subgroup and $\sigma_L$ is an irreducible supercuspidal $L$-representation. Let $P=MN$ be a standard parabolic subgroup of $G$ such that $L\subseteq M$. Set $J_M=J\cap M$ and $\tau_M=\tau|_{J_M}$. The setting of \cite[\S 8]{BK} applies, see the \cite[(9.5)]{BK}. In particular, by \cite[Proposition 8.5]{BK}, $(J,\tau)$ is a cover of $(J_M,\tau_M)$.

Let $\mathcal C(M,J_M,\tau_M)$ be the corresponding unipotent Bernstein block for $M$ and let $m_{\tau_M}$ denote the equivalence of categories with $\cH(M,J_M,\tau_M)$-mod. By \cite[(7.12)]{BK}, there exists an injective algebra homomorphism $\cH(M,J_M,\tau_M)\hookrightarrow \cH(G,J,\tau)$ and define the parabolic induction at the level of Hecke algebras accordingly
\[
\Ind_{\cH(M,J_M,\tau_M)}^{\cH(G,J,\tau)}(-)=\cH(G,J,\tau)\otimes_{\cH(M,J_M,\tau_M)}(-).
\]
By \cite[Corollary 8.4]{BK}, there is a commutative diagram
\begin{equation}\label{e:BK}
  \xymatrix@+1pc{
    {\mathcal C(G,J,\tau)} \ar[r]^{m_\tau} 
    & {\cH(G,J,\tau)\text{-mod}}\\
    {\mathcal C(M,J_M,\tau_M)  }\ar[u]^{\Ind_P^G} \ar[r]_{m_{\tau_M}}
    & {  \cH(M,J_M,\tau_M)\text{-mod}}\ar[u]_{\Ind_{\cH(M,J_M,\tau_M)}^{\cH(G,J,\tau)} }}
\end{equation}

Let $K$ be a maximal compact open subgroup containing $J$. If $\mu$ is any irreducible finite-dimensional $K$-representation, set
$\mu^\tau=\Hom_J(\tau,\mu),$ 
a module for the subalgebra $\cH(K,J,\tau)$ of functions is $\cH(G,J,\tau)$ supported in $K$. By \cite[Lemma 4.2]{Re2}
\begin{equation}
    \langle \mu,V^{U_K}\rangle_K=\langle \mu^\tau,m_\tau(V)\rangle_{\cH(K,J,\tau)},
\end{equation}
where $\langle ~,~\rangle$ denotes the appropriate multiplicity. Suppose $V=i_P^G(V_0)$, where $V_0$ is an admissible $M$-representation. Combining with (\ref{e:BK}), we get
\[
 \langle \mu,V^{U_K}\rangle_K=\langle \mu^\tau,\Ind_{\cH(M,J_M,\tau_M)}^{\cH(G,J,\tau)}m_{\tau_M}(V_0)\rangle_{\cH(K,J,\tau)}.
\]
The structure of the algebra $\cH(G,J,\tau)$ is known from the works of Morris and Lusztig. The important point for us is that there exists an (extended) affine Weyl group $\wti W=W\ltimes X$, $W$ a finite Weyl group, $X$ a $\bZ$-lattice, with a set of simple affine reflections $\wti S$ and a length function $\ell$ such that a $\bC$-basis of $\cH(G,J,\tau)$ is given by $\{T_w:w\in \wti W\}$. The relations between the generators are
\begin{align*}
T_w\cdot T_{w'}=T_{ww'},\text{ if }\ell(ww')=\ell(w)+\ell(w'),\\
(T_s+1)(T_s-q^{c(s)})=0,\ s\in \wti S,
\end{align*}
for some positive integers $c(s)$. The subalgebra $\cH(K,J,\tau)$ is the finite Hecke algebra generated by a finite parahoric subgroup $\wti W_K$ of $\wti W$. Let $W_K\subseteq W$ be the image of $\wti W_K$ under the projection map $\wti W\twoheadrightarrow W$.

The algebra $\cH(M,J_M,\tau_M)$ corresponds to the parabolic affine Hecke algebra generated by $\wti W_M=W_M\ltimes X$ for a parabolic subgroup $W_M$ of $W$.

Let $\cH(\wti W,c,v)$ denote the affine Hecke algebra with the same generators and relations, but where $q$ is replaced by $v^2$, where $v$ is an invertible indeterminate. This is a generic affine Hecke algebra in the sense of Lusztig. If $X_v$ is any $\cH(\wti W,c,v)$-module, write $X_{\sqrt q}=X_v\otimes_{\bC[v,v^{-1}]} \bC_{\sqrt q}$ for the specialization where $v$ acts by $\sqrt q$. Then $X_{\sqrt q}$ is a module for $\cH(G,J,\tau)$. Similarly, let $X_1$ be the specialization $v\to 1$, a module for $\bC[\wti W]$.

If $V$ is any irreducible tempered representation of $G$, $m_\tau(V)$ is an irreducible tempered $\cH(G,J,\tau)$-module or zero. By Lusztig's results, there exists a module $X_v$ for the generic affine Hecke algebra such that $m_\tau(V)=X_{\sqrt q}$. Denote, as in \cite[(6.1)]{Re2}
\[
m_\tau(V)_{q\to 1}=X_1, \text{ a }\wti W-module.
\]
As in \cite[\S 6.1]{Re2}, it follows that
\[\langle \mu^\tau,\Ind_{\cH(M,J_M,\tau_M)}^{\cH(G,J,\tau)}m_{\tau_M}(V_0)\rangle_{\cH(K,J,\tau)}=\langle (\mu^\tau)_{q\to 1}, \Ind_{\wti W_M}^{\wti W} (m_{\tau_M}(V_0))_{q\to 1})\rangle_{\wti W_K}.
\]
Finally, by the Mackey formula \cite[(6.2)]{Re2}
\[
\Ind_{\wti W_M}^{\wti W}(Y)|_{\wti W_K}=\bigoplus_{w\in W_K\backslash W/W_M} \Ind_{\wti W_M^w\cap \wti W_K}^{\wti W_K}(Y^w).
\]
Putting these formulae together, we find
\begin{equation}\label{e:mackey}
    \langle \mu,V^{U_K}\rangle_K=\bigoplus_{w\in W_K\backslash W/W_M} \langle (\mu^\tau)_{q\to 1},\Ind_{\wti W_M^w\cap \wti W_K}^{\wti W_K}(m_{\tau_M}(V_0))^w_{q\to 1})\rangle_{\wti W_K}.
\end{equation}

\section{Elliptic and compact pairs}\label{s:ellipticpairs}

\subsection{Definitions}\label{s:ell-pairs} Suppose $\Gamma$ is a (possibly disconnected) complex reductive group with identity component $\Gamma^\circ$. We define the sets (as in \cite[Def.~1.1]{Ciu}):
\begin{equation}\label{e:Y-gamma}
\begin{aligned}
\cY(\Gamma)&=\{(s,h)\in \Gamma\times\Gamma\mid s,h \text{ semisimple}, \ sh=hs\},\\
\cY(\Gamma)_\ellip&=\{(s,h)\in \Gamma\times\Gamma\mid s,h \text{ semisimple}, \ sh=hs, \ \rZ_{\Gamma}(s,h)\text{ is finite}\}.
\end{aligned}
\end{equation}
Here $\rZ_{\Gamma}(s,h)=\rZ_\Gamma(s)\cap \rZ_\Gamma(h)$.
We refer to elements of $\cY(\Gamma)_\ellip$ as {\it elliptic pairs.} 

We extend this definition in the following way.

\begin{defn}
Suppose $\Gamma'$ is a reductive subgroup of $\Gamma$ containing $\rZ_\Gamma$. We say that a pair $(s, h) \in \cY(\Gamma')$ is \textit{essentially $\Gamma$-elliptic} if $\rZ_{\Gamma'}(s, h)$ is finite modulo $\rZ_\Gamma$. Write $\cY(\Gamma')_{\Gamma-\ess}$ for the set of essentially $\Gamma$-elliptic pairs in $\Gamma'$. 

\end{defn}
Clearly every elliptic pair in $\cY(\Gamma')$ is essentially $\Gamma$-elliptic.

Now define the relations on $\cY(\Gamma)$:
\begin{equation}
\begin{aligned}
&(s,h)\sim_L (s,h') \text{ if }\gamma h\gamma^{-1}\in h'T\text{ for some }\gamma\in \rZ_\Gamma(s) \text{ and } T\text{ a maximal torus in }\rZ_\Gamma(s,h);\\
&(s,h)\sim_R (s',h) \text{ if }\gamma_1 s\gamma_1^{-1}\in s'T\text{ for some }\gamma_1\in \rZ_\Gamma(h) \text{ and } T\text{ as before.}
\end{aligned}
\end{equation}

\begin{lem}\label{l:ell-comp} Fix $s\in \Gamma$ semisimple. The projection map $\rZ_\Gamma(s)\to A_\Gamma(s)$, $h\mapsto \bar h$, induces:
\begin{enumerate}
\item a bijection between $\sim_L$-classes of pairs $(s,h)\in \cY(\Gamma)$ and conjugacy classes in $A_\Gamma(s)$;
\item a bijection between $\rZ_\Gamma(s)$-orbits of elliptic pairs $(s,h)$ and the elliptic conjugacy classes in $A_\Gamma(s)$.
\end{enumerate}
\end{lem}
\begin{proof}
We need a result from the theory of semisimple automorphisms of reductive groups, e.g. \cite[Proposition 9]{Som1}: if $x,y$ are semisimple elements in a reductive group $\cG$ such that their images in the group of components $\cG/\cG^\circ$ are in the same conjugacy class, and $S$ is a maximal torus in $\rZ_\cG(x)$, then there exist $g\in \cG$ and $s\in S$ such that $g y g^{-1}=xs$.
Apply this to $\cG:=\rZ_\Gamma(s)$ (a reductive group), then (1) follows immediately (this is in fact the motivation for the definition $\sim$).

Part (2) is \cite[Lemma 8.2]{ACR}.
\end{proof}

Let $\sim$ be the equivalence relation on $\cY(\Gamma)$ generated by $\sim_L$ and $\sim_R$. 
It is not hard to check that if $(s, h) \in \cY(\Gamma)$ and $\gamma \in \Gamma$, then $(s, h) \sim (\gamma s \gamma^{-1}, \gamma h \gamma^{-1})$, so $\sim$ induces an equivalence relation (also denoted by $\sim$) 
on $\Gamma\backslash \cY(\Gamma)$. The subset $\cY(\Gamma)_\ellip$ of $\cY(\Gamma)$ is closed under $\sim$-equivalence and notice that
\[ \cY(\Gamma)_\ellip/_\sim=\Gamma\backslash \cY(\Gamma)_\ellip.
\]

Now we consider the relation $\sim$ on essentially $\Gamma$-elliptic pairs as defined above.

\begin{lem}\label{l:equiv-ess}
Let $\Gamma'$ be a reductive subgroup of $\Gamma$ containing $\rZ_\Gamma$. 
The subset $\cY(\Gamma')_{\Gamma-{\ess}}$ of $\cY(\Gamma')$ is closed under $\sim$-equivalence, and 
\begin{equation*}
 \cY(\Gamma')_{\Gamma-\ess}/_\sim \;= \;\Gamma'\backslash \cY(\Gamma')_{\Gamma-\ess}/(\rZ_\Gamma^\circ \times \rZ_\Gamma^\circ),
\end{equation*}
where $\Gamma'$ is acting by conjugation and $\rZ_\Gamma^\circ \times \rZ_\Gamma^\circ$ is acting by multiplication.
\end{lem}

\begin{proof}
Suppose $(s, h_1) \in \cY(\Gamma')_{\Gamma-\ess}, (s, h_2) \in \cY(\Gamma')$, and $(s, h_1) \sim_L (s, h_2)$. Then $\gamma h_1 \gamma^{-1} \in h_2T$ for some $\gamma \in \rZ_{\Gamma'}(s)$ and some maximal torus $T$ in $\rZ_{\Gamma'}(s, h_1)$. By assumption $T = \rZ_\Gamma^\circ$, so $h_2 = \gamma h_1z\gamma^{-1}$ for some $z \in \rZ_\Gamma^\circ$. Since $\gamma$ centralizes $s$, we have $\rZ_{\Gamma'}(s, h_2) = \gamma\rZ_{\Gamma'}(s, h_1)\gamma^{-1}$. Thus $(s, h_2) \in \cY(\Gamma')_{\Gamma-\ess}$. Similarly if $(s_1, h) \sim_R (s_2, h)$ and $(s_1, h) \in \cY(\Gamma')_{\Gamma-\ess}$, then $(s_2, h) \in \cY(\Gamma')_{\Gamma-\ess}$. Thus $\cY(\Gamma')_{\Gamma-\ess}$ is closed under $\sim$-equivalence.

We already saw that $\Gamma'$-conjugate elements of $\cY(\Gamma')$ are $\sim$-equivalent. Next let $(s, h) \in \cY(\Gamma')_{\Gamma-\ess}$ and $z_1, z_2 \in \rZ_\Gamma^\circ$. We claim that $(s, h) \sim (sz_1, hz_2)$. We have that $(s, h) \sim_L (s, hz_2)$ since $\rZ_\Gamma^\circ$ is a maximal torus in $\rZ_{\Gamma'}(s, h)$. Similarly $(s, hz_2) \sim_R (sz_1, hz_2)$.

To finish the proof we must show that if two elements of $\cY(\Gamma')_{\Gamma-\ess}$ are $\sim$-equivalent, then they are equal in the set $\Gamma'\backslash \cY(\Gamma')_{\Gamma-\ess}/(\rZ_\Gamma^\circ \times \rZ_\Gamma^\circ)$. So let $(s, h_1) \in \cY(\Gamma')_{\Gamma-\ess}, (s, h_2) \in \cY(\Gamma')$, and suppose $(s, h_1) \sim_L (s, h_2)$. By the same logic as above, this tells us that there exists $\gamma \in \rZ_{\Gamma'}(s)$ and $z \in \rZ_{\Gamma}^\circ$ such that $\gamma h_1 \gamma^{-1} = h_2z$. Then in the set $\Gamma'\backslash \cY(\Gamma')_{\Gamma-\ess}/(\rZ_\Gamma^\circ \times \rZ_\Gamma^\circ)$, we have $(s, h_1) = (\gamma s \gamma^{-1}, \gamma h_1 \gamma^{-1}) = (s, h_2z) = (s, h_2)$. A similar argument for $\sim_R$ finishes the proof. \end{proof}

\begin{rem}
In every $\sim$-equivalence class of $\cY(\Gamma)$, we may choose (not uniquely) a representative $(s,h)$ such that both $s,h$ have finite order by \cite[Lemma 1.15]{DM2}. We call these representatives \emph{compact pairs}.
\end{rem}

\subsection{Relating compact and essentially elliptic pairs}

Let $G^\vee$ be a connected complex reductive group, and let $G^\vee_{\un}$ be the set of unipotent elements in $G^\vee$. We now relate compact pairs for $G^\vee$ to essentially elliptic pairs in Levi subgroups of $G^\vee$.

We describe two sets with $G^\vee$-actions. First, let $u \in G^\vee$ be unipotent. Let $\Gamma_u = \rZ_{G^\vee}(u)^{\red}$. Consider the set 
$\cup_{u \in G^\vee_{\un}} \cY(\Gamma_u)$.
Note that $G^\vee$ acts on this set by conjugation: $g \cdot (u, s, h) = (gug^{-1}, gsg^{-1}, ghg^{-1})$. 
It's easy to check that the action of $G^\vee$ by conjugation induces an action on the set 
\begin{equation*}
\cpt(G^\vee) := \bigcup_{u \in G^\vee_{\un}} \cY(\Gamma_u)/_\sim
\end{equation*}
 of $\sim$-equivalence classes.

Let $\cL^\vee$ be the set of Levi subgroups of $G^\vee$. Next we consider the set
\begin{equation*}
\ess(G^\vee) := \bigcup_{M^\vee \in \mathcal{L}^\vee} \bigcup_{u \in M^\vee_{\un}} \cY(\Gamma_u^{M^\vee})_{M^\vee-\ess} /(\rZ_{M^\vee}^\circ \times \rZ_{M^\vee}^\circ)
 \end{equation*}
where $\Gamma_u^{M^\vee} = \Gamma_u \cap M^\vee$. 
Clearly $G^\vee$ acts on the set of tuples $(M^\vee, u, s, h) \in \cup_{M^\vee} \cup_u \cY(\Gamma_u^{M^\vee})_{M^\vee-\ess}$ 
by conjugation. It's also not hard to check that this induces a well-defined action on $\ess(G^\vee)$.

Now we define a map between the two sets just described. Let $u \in G^\vee$ be unipotent, and let $(s, h) \in \cY(\Gamma_u)$. Given this triple $(u, s, h)$, choose a maximal torus $S \subset (\rZ_{\Gamma_u}(s, h))^\circ$. Then $M^\vee := \rZ_{G^\vee}(S)$ is a Levi subgroup of $G^\vee$. Note that $u \in M^\vee$ (since $S \subset \Gamma_u$), and $\rZ_{M^\vee}(u)^{\red} = \Gamma_u \cap M^\vee$. Thus $s, h \in \Gamma_u^{M^\vee} := \rZ_{M^\vee}(u)^{\red}$. We claim that $(s, h)$ is an essentially $M^\vee$-elliptic pair in $\Gamma_u^{M^\vee}$. This follows from the fact that if $S'$ is a torus in $\rZ_{\Gamma_u^{M^\vee}}(s, h)$, then $S' \subset S \subset \rZ_{M^\vee}^\circ$, so $(s, h)$ is elliptic modulo $\rZ_{M^\vee}^\circ$. We have thus described a map
\begin{equation}\label{e:kappa}
\kappa:\bigcup_{u \in G^\vee_{\un}} \cY(\Gamma_u) \longrightarrow \ess(G^\vee)/G^\vee, \quad 
   (u,s,h)\mapsto (M^\vee,u,s,h),
\end{equation}
where we are taking $G^\vee$-orbits on $\ess(G^\vee)$.

Note that given a triple $(u, s, h)$, the image $\kappa((u, s, h))$ does not depend on choice of maximal torus $S \subset (\rZ_{\Gamma_u}(s, h))^\circ$: if $S' \subset (\rZ_{\Gamma_u}(s, h))^\circ$ is another maximal torus, then there exists $z \in \rZ_{\Gamma_u}(s, h)^\circ$ such that $S' = zSz^{-1}$. The corresponding Levi $\rZ_{G^\vee}(S'$) is equal to $z\rZ_{G^\vee}(S)z^{-1}$, and so the image in $\ess(G)/G^\vee$ is independent of $S$. 

\begin{lem}
The map $\kappa$ gives a well-defined function on $\cpt(G^\vee)$.
\end{lem}

\begin{proof}
We must show that $\kappa$ is well defined on $\sim$-equivalence classes. Fix $u \in G^\vee_{\un}$, let $(s, h_1), (s, h_2) \in \cY(\Gamma_u)$, and suppose $(s, h_1) \sim_L (s, h_2)$. Let $H = \rZ_{\Gamma_u}(s)$. 
By definition of $\sim_L$ there exists $\gamma \in H$ and a maximal torus $S_1 \subset \rZ_H(h_1)$ such that $\gamma h_1 \gamma^{-1} \in h_2 S_1$. The image of $(s, h_1)$ under the map described above is thus $(M^\vee_1 := \rZ_{G^\vee}(S_1), u, s, h_1)$.
Since the image of $(s, h_1)$ is well defined up to conjugation by $\Gamma_u$ and $\gamma$ centralizes $s$, we may replace $(s, h_1)$ with $(s, \gamma h_1 \gamma^{-1})$, and so we may assume that $h_1 \in h_2S_1$. 

It is then easy to check that $S_1 \subset \rZ_H(h_2)$, and so we may choose a maximal torus $S_2 \subset \rZ_{H}(h_2)$ containing $S_1$. Again using the fact that $h_1 \in h_2S_1$, one can check that $S_2 \subset \rZ_H(h_1)$. By maximality of $S_1$ we then have $S_2 \subset S_1$, so in fact $S_1 = S_2$. The image of $(s, h_2)$ is thus $(M_1^\vee, u, s, h_2)$. But $S_1 \subset \rZ_{M^\vee_1}^\circ$, so $h_1 = h_2 \mod \rZ_{M^\vee_1}^\circ$, which shows that the pairs $(s, h_1)$ and $(s, h_2)$ are equal in $\cY(\Gamma_u^{M^\vee})_{M^\vee-\ess}/(\rZ_{M^\vee}^\circ \times \rZ_{M^\vee}^\circ)$. Thus $\kappa$ is well defined on $\sim_L$-equivalence classes. The proof that $\kappa$ is well defined on $\sim_R$-equivalence classes is the same.
\end{proof}

\begin{lem}\label{l:levi-center}
Let $u \in G^\vee, (s, h) \in \cY(\Gamma_u), S \subset \rZ_{\Gamma_u}(s, h)$, and $M^\vee = \rZ_{G^\vee}(S)$ be as defined above. Then $Z_{M^\vee}^\circ = S$.
\end{lem}

\begin{proof}
By definition of $M^\vee$ we have $S \subset \rZ_{M^\vee}^\circ$. Since $u, s, h \in M^\vee$, we have $\rZ_{M^\vee}^\circ \subset \rZ_{\Gamma_u}(s, h)$, so we may choose a maximal torus $S'$ of $\rZ_{\Gamma_u}(s, h)$ such that $\rZ_{M^\vee}^\circ \subset S'$. By the maximality of $S$, we must have $S = S' = \rZ_{M^\vee}^\circ$.
\end{proof}

\begin{prop}\label{prop-map-on-sets}
The map
$\cpt(G^\vee)/G^\vee \longrightarrow \ess(G^\vee)/G^\vee$
on $G^\vee$-orbits induced by $\kappa$, as defined in (\ref{e:kappa}), is a bijection.
\end{prop}

\begin{proof}
 It is easy to see that $\kappa$ is well defined on $G^\vee$-orbits.

To show that $\kappa$ is surjective, let $M^\vee \in \cL^\vee$, $u \in M^\vee_{\un}$, and $(s, h) \in \cY(\Gamma_u^{M^\vee})_{M^\vee-\ess}$. We claim that $(M^\vee, u, s, h)$ is the image of $(u, s, h)$. To see this, note that $u, s, h \in M^\vee$, so $\rZ_{M^\vee}^\circ \subset \rZ_{\Gamma_u}(s, h)$. Thus we can choose a maximal torus $S \subset \rZ_{\Gamma_u}(s, h)$ containing $\rZ_{M^\vee}^\circ$. Let $(M')^\vee = \rZ_{G^\vee}(S)$
 It suffices to show that $M^\vee = (M')^\vee$. Since $S = \rZ_{(M')^\vee}^\circ$ by Lemma \ref{l:levi-center}, we have $\rZ_{M^\vee}^\circ \subset \rZ_{(M')^\vee}^\circ$ and $(M')^\vee \subset M^\vee$. This tells us that $\rZ_{\Gamma_u}(s, h) \cap (M')^\vee \subset \rZ_{\Gamma_u}(s, h) \cap M^\vee$, so the maximal torus $\rZ_{(M')^\vee}^\circ$ of $\rZ_{\Gamma_u}(s, h) \cap (M')^\vee$ is contained in a maximal torus of $\rZ_{\Gamma_u}(s, h) \cap M^\vee$. Since $(s, h)$ is essentially elliptic in ${\Gamma_u^{M^\vee}}$, we know that $\rZ_{M^\vee}^\circ$ is a maximal torus of $\rZ_{\Gamma_u}(s, h) \cap M^\vee$. So we must have the containment $\rZ_{(M')^\vee}^\circ \subset \rZ_{M^\vee}^\circ$, which gives $M^\vee = (M')^\vee$.

To show that $\kappa$ is injective, suppose $(u, s, h)$ and $(u', s', h')$ have the same image. If we choose maximal tori $S \subset \rZ_{\Gamma_u}(s, h)$ and $S' \subset \rZ_{\Gamma_{u'}}(s', h')$, this tells us that the centralizers $\rZ_{G^\vee}(S)$ and $\rZ_{G^\vee}(S')$ are $G^\vee$-conjugate, so after conjugating we may assume they are equal; denote this Levi $M^\vee$. Now by assumption $u$ and $u'$ are conjugate in $N_{G^\vee}(M^\vee)$, so after conjugating by $N_{G^\vee}(M^\vee)$, we may assume $u = u'$ and that $(s, h) = (s', h')$ in $\cY(\Gamma_u^{M^\vee})_{M^\vee-\ess}/(\rZ_{M^\vee}^\circ \times \rZ_{M^\vee}^\circ)$,
Then $(s, h) \sim (s', h')$ as elements of $\cY(\Gamma_u)$, by Lemma \ref{l:equiv-ess}.
\end{proof}

\begin{lem}
Suppose $(s, h) \in \cY(\Gamma_u)$ is conjugate to $(s', h') \in \cY(\Gamma_u)$ under the action of $\rZ_{G^\vee}(u) \cap N_{G^\vee}(\Gamma_u)$. Then $(s, h)$ is conjugate to $(s', h')$ under the action of $\Gamma_u$.
\end{lem}

\begin{proof} 
We have that $\rZ_{G^\vee}(u) = \Gamma_uN$ where $N$ is a unipotent subgroup normal in $\rZ_{G^\vee}(u)$. Let $N' = \{n \in N \mid n\Gamma_un^{-1} = \Gamma_u\}$. Then $\rZ_{G^\vee}(u) \cap N_{G^\vee}(\Gamma_u) = \Gamma_uN'$. It suffices to show that $N'$ centralizes $\Gamma_u$, but this follows from the fact that $\Gamma_u$ and $N'$ are both normal in $\Gamma_uN'$. 
\end{proof}

\begin{cor}\label{cor-cpt-ell}
The map $\kappa$ of Proposition \ref{prop-map-on-sets} induces a linear isomorphism
\begin{equation*}
\bigoplus_{u \in G^\vee_{\un}/G^\vee} \bC[\cY(\Gamma_u)/\sim]^{\Gamma_u} \longrightarrow \bigoplus_{M^\vee \in \mathcal{L}^\vee/G^\vee} \bigoplus_{u \in M^\vee_{\un}/M^\vee} \bC[\cY(\Gamma_u^{M^\vee})_{M^\vee-\ess}/(\rZ_{M^\vee}^\circ \times \rZ_{M^\vee}^\circ)]^{N_{\Gamma_u}(M^\vee)}
\end{equation*}

\end{cor}
\begin{proof}
    This follows immediately from Proposition \ref{prop-map-on-sets} by taking functions on both sides of the bijection $\kappa$. Note that we must use the previous lemma when taking $\Gamma_u$-invariants on the lefthand side.
    To explain the appearance of $N_{\Gamma_u}(M^\vee)$, we can use a similar logic as in the previous lemma to show that conjugation of $\Gamma_u^{M^\vee}$ by $N_{G^\vee}(M^\vee) \cap \rZ_{G^\vee}(u) \cap N_{G^\vee}(\Gamma_u^{M^\vee})$ is equivalent to conjugation by $N_{\Gamma_u}(M^\vee)$.
\end{proof}

\begin{ex}\label{ex-pairs}
\begin{enumerate}
    \item If $u$ is regular, then $\Gamma_u = \rZ_{G^\vee}$, the corresponding maximal torus is $S=\rZ_{G^\vee}^\circ$, the corresponding Levi subgroup $M^\vee$ is $G^\vee$, and $(s, h)$ is already essentially $G^\vee$-elliptic in $\Gamma_u^{M^\vee} = \Gamma_u=\rZ_{G^\vee}$. For this choice of $u$, $\kappa$ identifies compact pairs in $\rZ_{G^\vee}$ with elliptic pairs in the component group $\rZ_{G^\vee}/\rZ_{G^\vee}^\circ$. 
\item 
For every unipotent $u \in G^\vee$, we may consider the pair $(s, h) = (1, 1) \in \cY(\Gamma_u)$. In this case Proposition \ref{prop-map-on-sets} tells us that, up to $G^\vee$-conjugacy, there exists a unique Levi subgroup $M^\vee \subset G^\vee$ containing $u$ such that $(\Gamma_u \cap M^\vee)^\circ = \rZ_{M^\vee}^\circ$ (i.e. such that $u$ is distinguished in $M^\vee$). Further, $\kappa$ gives a bijection between unipotent classes in $G^\vee$ and conjugacy classes of pairs $(M^\vee, u_d)$ where $M^\vee$ is a Levi subgroup of $G^\vee$ and $u_d$ is a distinguished unipotent element of $M^\vee$. We have thus recovered a part of the well-known classification of Bala--Carter (cf. \cite[\S8]{CM}).
\end{enumerate}
\end{ex}

\section{The essentially elliptic Fourier transform}\label{s:ess-ell}

We now extend the results of \cite{ACR} from the elliptic case to the essentially elliptic case. 
We first recall the definition of the linear combinations $\Pi(u, s, h)$ (cf. \cite[Section 9.1]{ACR}) and prove a basic result that will relate these combinations for $G$ to those for its Levi subgroups.

Given a unipotent element $u \in G^\vee$ and a compact pair $(s, h) \in \cY(\Gamma_u)$, let
\begin{equation} \label{eqn:pi_ush}
    \Pi(u, s, h) = \sum_{\rho \in \widehat{A}_{\Gamma_{u}}(s)} \rho(h)\pi(u, s, \rho) \in \bigoplus_{G' \in \InnT^p(G)} R_{\un}(G').
\end{equation}
If we want to specify the group $G$, we write $\Pi^G(u, s, h)$ for $\Pi(u, s, h)$.
Note that $(s, h) \sim_L (s, h')$, then $h$ and $h'$ are conjugate in $A_{\Gamma_u}(s)$ (see Lemma \ref{l:ell-comp}), and therefore $\Pi(u, s, h)=\Pi(u, s, h')$.

\begin{rem} If $(s, h) \sim_R (s', h)$, $\Pi(u, s, h)\neq \Pi(u, s', h)$ in general. For example, if $u=1$ and $G^\vee$ is simply connected, then $\Pi(1,s,1)=\pi(1,s,\mathbf 1)$. This is precisely the unramified minimal principal series with Satake parameter $s$. Of course, $\pi(1,s,\mathbf 1)\neq \pi(1,s',\mathbf 1)$ if $s$ and $s'$ are not conjugate. However, we'll later show that the images of of $\Pi(u, s, h)$ and $\Pi(u, s', h')$ are same in the space $\Rcpt(G)$, which will be defined in Section \ref{s:rigid}.
\end{rem}

In the following lemma, as in Section \ref{s:llc}, given an enhanced Langlands parameter $(\varphi_{(u, s)}, \phi)$ for a split connected reductive group $M$, we write $M_\phi$ for the pure inner twist of $M$ such that $\pi(u, s, \phi)$ is a representation of $M_\phi$. 

\begin{lem}\label{l:ind-Pi}
Let $M^\vee$ be a Levi subgroup of $G^\vee$, let $u \in M^\vee$ be unipotent, and suppose $(s, h) \in \cY(\Gamma_u)_{M^\vee-\ess}$ with $s$ of finite order. Then 
\begin{equation*}
\sum_{\varphi \in \widehat{A}_{\Gamma_u^{M^\vee}}(s)} \phi(h) i_{M_\phi}^{G_{\phi}} \pi^{M_{\phi}}(u, s, \phi) = \Pi(u, s, h),
\end{equation*}
where for each $\phi$, $G_\phi$ is the image of $M_\phi$ under (\ref{eqn-Inn-Levi}). 
In other words, in the notation of Remark \ref{r:par-ind-twists}, we have $i_M^G (\Pi^M(u, s, h)) = \Pi(u, s, h)$.
\end{lem}

\begin{proof} For simplicity of notation, write $A_M$ for $A_{\Gamma_u^{M^\vee}}(s)$ and $A_G$ for ${A}_{\Gamma_u^{G^\vee}}(s)$.
By \cite[Lemma 5.9]{So1}, we have 
\begin{equation*}
i_{M_\phi}^{G_{\phi}} \pi^{M_\phi}(u, s, \phi) = \sum_{\rho \in \widehat{A}_G} \la \rho, \phi\ra_{A_M} \pi^{G}(u, s, \rho),
\end{equation*}
where $\langle \cdot, \cdot \ra_{A_M}$ is the character pairing on $A_M$.
Thus we have
\begin{align*}
\sum_{\phi \in \widehat{A}_{M}} \phi(h)i_{M_\phi}^{G_{\phi}} \pi^{M_{\phi}}(u, s, \phi) &= \sum_{\phi \in \widehat{A}_M} \sum_{\rho \in \widehat{A}_G} \phi(h)\langle \rho, \phi\rangle_{A_M} \pi^{G_\rho}(u, s, \rho)\\
&= \sum_{\rho \in \widehat{A}_G} (\sum_{\phi \in \widehat{A}_{M}} \phi(h)\langle \rho, \phi\rangle )\pi^{G_\phi}(u, s, \rho)\\
&= \sum_{\rho \in \widehat{A}_G} \rho(h)\pi(u, s, \rho)\\
&= \Pi(u, s, h)
\end{align*}
\end{proof}

\subsection{The essential elliptic space}
Let 
\begin{equation*}
\Ell(G^\vee) = \bigoplus_{u \in G^\vee_\un/G^\vee} \bC[\cY(\Gamma_u)_{G^\vee-\ess}]^{\Gamma_u}
\end{equation*}
where 
the action of $\Gamma_u$ is by conjugation. Let
\begin{equation*}
\Ress(G) = \text{span}_{\bC}\{\Pi(u, s, h) \mid (s, h) \in \underset{u \in G^\vee_{\un}}\bigcup \cY(\Gamma_u)_{G^\vee-\ess}\}.
\end{equation*}
Note that $\rZ_{G^\vee}^\circ$ acts on $\cY(\Gamma_u)_{G^\vee-\ess}$ via $z \cdot (s, h) = (s, zh)$. If $z \in \rZ_{G^\vee}^\circ$, then $\Pi(u, s, zh) = \Pi(u, s, h)$. So we have a well-defined map 
\begin{equation*}
\tau: \Ell(G^\vee)/\rZ_{G^\vee}^\circ \to \Ress(G)
\end{equation*}
given by sending $(s, h) \in \cY(\Gamma_u)_{G^\vee-\ess}$ to $\Pi(u, s, h)$.

As in Section \ref{s:ind-res-prelim}, for every rational parabolic subgroup $P=MN$ of $G'\in\InnT^p(G)$, let $i_P^{G'}: R_\un(M)\to R_\un(G')$ denote the parabolic induction map. Let $\overline R_\un(G')$ denote the {\it elliptic representation space} of $G'$, that is, the quotient of $R_\un(G)$ by the subspace of proper parabolically induced representations $R_\un^\ind(G')=\sum_{P\neq G'} i_P^{G'}(R_\un(M))$. Set
\[
\Ell_\un^p(G)=\bigoplus_{G'\in \InnT^p(G)} \overline R_\un(G').
\]

In the following theorem, we write $p_e: \oplus_{G' \in \InnT^p(G)} R_\un(G') \to \Ell^p_{\un}(G)$ for the natural projection map.
\begin{thm}\label{t:ess}
The map $\tau$ is a linear isomorphism. Moreover, the set 
     \begin{equation*}
     \{p_e(\Pi(u, s, h)) \mid (s, h) \in \bigcup_{u \in G^\vee_{\un}/G^\vee} \Gamma_u\backslash\cY(\Gamma_u)_{G^\vee-\ess}\} 
     \end{equation*}
     forms a basis of $\Ell_\un^p(G)$,  hence
 \[
 \Ell_\un^p(G)\cong \Ress(G)\cong \Ell(G^\vee)/\rZ_{G^\vee}^\circ. 
 \]
\end{thm}

\begin{proof} This is essentially in \cite{ACR}, we only explain the small differences.

    When $G$ is semisimple and adjoint, this is proved in \cite[Theorem 11.1]{ACR}. Firstly, we note that in fact the restriction to ``adjoint'' in \cite{ACR} is not necessary. The key step is provided by \cite[Proposition 11.11]{ACR}, where Waldspurger's elliptic pairing theorem and Clifford theory are used to obtain a correspondence between the elliptic space for the representations of the relevant finite Weyl groups and the elliptic space of component groups $A(u)$. As noted in \cite[Remark 11.12]{ACR}, the issue is that the elliptic pair for the component group that occurs in the proposition is $(\phi,\phi')^{A_u(j)}_{\text{ell}}$ rather than $(\phi,\phi')^{A_u}_{\text{ell}}$ as expected. However, by \cite[Lemma 4.4]{AMS1}, the supports of $\phi,\phi'$ are contained in $A_u(j)$, implying that $(\phi,\phi')^{A_u}_{\text{ell}}=\frac{|A_u(j)|}{|A_u|}(\phi,\phi')^{A_u(j)}_{\text{ell}}$, whence $(\phi,\phi')^{A_u}_{\text{ell}}$ is a nonzero scalar multiple of $(\mathbf \rho,\mathbf \rho')^{W_j}_{\text{ell}}$, with the notation as in {\it loc. cit.}. The rest of the proof in the semisimple case is identical with the adjoint case, hence \cite[Theorem 11.1]{ACR} holds without an assumption on the isogeny, except that the isomorphism in the theorem is not necessarily a isometry (but a kind of homothety).

    Secondly, to extend from the semisimple to the reductive case, we note that the proofs in {\it loc. cit.} hold as written with two modifications:
    \begin{enumerate}
    \item[(a)] If $x\in G^\vee$, the elliptic theory of the finite group $A_x=Z_{G^\vee}(x)/Z_{G^\vee}^\circ$ is defined with respect to the linear representation $\mathfrak s_x/\mathfrak z_{G^\vee}$ rather than $\mathfrak s_x$, where $\mathfrak s_x$ is the abstract Cartan subalgebra of $Z_{G^\vee}^\circ(x)$ and $\mathfrak z_{G^\vee}$ is the Lie algebra of $\rZ_{G^\vee}^\circ$.
    \item[(b)] If $M^\vee$ is a Levi subgroup of $G^\vee$ (or more generally of a reductive group $H^\vee=Z_{G^\vee}(s)$ that appears in the classification) that has an equivariant cuspidal local system supported on a nilpotent orbit, then the elliptic theory of the finite Coxeter group $W_{M^\vee}=N_{H^\vee}(M^\vee)/M^\vee$ is defined with respect to the linear representation $\mathfrak z_{M^\vee}/\mathfrak z_{G^\vee}$, rather than $\mathfrak z_{M^\vee}$, where $\mathfrak z_{M^\vee}$ is the Lie algebra of $\rZ_M^\circ$.
    \end{enumerate}
    Notice that the key result that one uses in the proof, namely \cite[Theor\'eme]{Wa}, holds in this generality.
\end{proof}

We can endow $\Ell^p_\un(G)$ with a symmetric bilinear form:
\begin{equation}
    \EP_{\der}(U,V)=\sum_{i\ge 0}(-1)^i \Ext^i_{G_\der}(U,V).
\end{equation}
The reason for considering the Euler--Poincar\'e pairing with respect to the derived subgroup $G_\der$ is that, when $G$ is not semisimple, the usual Euler--Poincar\'e pairing with respect to $G$ is identically zero.
Let $\Psi(G)$ be the group of unramified characters of $G$. It is clear that every $V-V\otimes \chi$, $\chi\in \Psi(G)$, is in the kernel of $\EP_{\der}$, hence the pairing descends naturally to the space of coinvariants $\Ell^p_\un(G)_{\Psi(G)}$.

\medskip

Note that $\rZ_{G^\vee}^\circ$ also acts on $\Ell(G^\vee)$ by $z \cdot (s, h) = (zs, h)$. We have that $\rZ_{G^\vee}^\circ \simeq \Psi(G)$, and $\Psi(G)$ acts on $\Ress(G)$ by twisting by unramified characters. The map $\tau$ above is equivariant for these two actions of $\rZ_{G^\vee}^\circ$. Then Theorem \ref{t:ess} immediately gives:

\begin{cor}\label{c:ess-ell}
   The map $\tau$ induces an isomorphisms $\Ell^p_\un(G)_{\Psi(G)}\cong \Ell(G^\vee)/(\rZ_{G^\vee}^\circ \times \rZ_{G^\vee}^\circ)$. We have a commutative diagram
    \begin{equation}\label{e:der}
    \begin{tikzcd}
\Ell^p_\un(G)_{\Psi(G)} \arrow{r}{\tau_G} \arrow[swap]{d}{r} & \Ell(G^\vee)/(\rZ_{G^\vee}^\circ \times \rZ_{G^\vee}^\circ) \arrow{d}{p} \\%
\Ell^p_\un(G_\der) \arrow{r}{\tau_{G_\der}}& \Ell(G^\vee/\rZ_{G^\vee}^\circ)
\end{tikzcd}
\end{equation}
where $r$ is the restriction of representations and $p$ is the natural map $(s,h)\mapsto (s\rZ_{G^\vee}^\circ,h\rZ_{G^\vee}^\circ)$.
\end{cor}

\begin{rem}
    The vertical arrows in (\ref{e:der}) are not isomorphisms in general. Consider for example $G=\GL_2(F)$, $G_\der=\SL_2(F)$. Then the two spaces in the top row are one dimensional, while the two spaces in the bottom row are two dimensional.
\end{rem}

\subsection{The involution}
Define $\FT^\vee_{\ess}: \Ress(G)_{\Psi(G)} \to \Ress(G)_{\Psi(G)}$ as the linear map given on basis elements by \[\Pi(u, s, h) \mapsto \Pi(u, h, s).\]
Note that $\FT^\vee_{\ess}$ is well defined since $\Pi(u, s, zh) = \Pi(u, s, h)$ for all $z \in \rZ_{G^\vee}^\circ$. 

As in \cite[Section 7]{ACR}, given $G' \in \InnT^p(G)$, we let $\max(G')$ denote a set of conjugacy-class representatives of maximal compact open subgroups of $G$. 
    Given a maximal compact subgroup $K$ of $G$, we defined in \cite[(6.3)]{ACR} an involution 
  \begin{equation*}
  \FT_K : \bigoplus_{H \in \InnT(\overline{K})} R_\un(H) \to \bigoplus_{H \in \InnT(\overline{K})} R_\un(H),
  \end{equation*}
  which is Lusztig's nonabelian Fourier transform in the case when $\overline{K}$ is connected.
Letting
\begin{equation*} 
  \cC(G)_{\cpt,\un} = \bigoplus_{G'\in \InnT^p G}~~\bigoplus_{K'\in \max(G')} R_\un (\overline{K'}) = \bigoplus_{K \in \max(G)} \bigoplus_{H \in \InnT(\overline{K})} R_\un(H),
  \end{equation*}
we define 
  \begin{equation*}
  \FT_{\cpt, \un}: \cC(G)_{\cpt,\un} \to \cC(G)_{\cpt,\un}
  \end{equation*}
  by taking $\FT_K$ on the summand corresponding to $K$. This is
 the map defined in \cite[Definition 7.1]{ACR}.

Since weakly unramified characters are trivial on compact subgroups, parahoric restriction gives a well-defined map 
\begin{equation*}
\res_{\cpt,\un}: \Ress(G)_{\Psi(G)} \to \cC(G)_{\cpt,\un}.
\end{equation*}
The natural extension of \cite[Conjecture 9.7]{ACR} in this setting is as follows:
\begin{conj}\label{c:elliptic} Let $\mathbf{M}$ be a connected reductive $F$-split group, and let $M = \bM(F)$. Consider the following diagram:
  \begin{displaymath}\xymatrix@+1pc{
    {\Ress(M)_{\Psi(M)}} \ar[r]^{\FT_{\ess}^\vee} \ar[d]_{\res_{\cpt,\un}}
    & {\Ress(M)_{\Psi(M)}}\ar[d]^{\res_{\cpt,\un}}\\
    {  \cC(M)_{\cpt,\un}} \ar[r]_{\mathrm{FT_{\cpt,\un}}}
    & {  \cC(M)_{\cpt,\un}} }
\end{displaymath} 
For every  unipotent element $u\in M^\vee$, essentially $M^\vee$-elliptic pair $(s,h) \in \cY(\Gamma_u)_{M^\vee-\ess}$, and maximal compact open subgroup $K$ of $M$, there exists a root of unity $\zeta=\zeta(u,s,h,K)$ such that
\[\res_K (\Pi^M(u,h,s))=\zeta \cdot (\FT_{\cpt,\un}\circ \res_K)(\Pi^M(u,s,h)). \]
\end{conj}
(Note that we have stated the conjecture in terms of $M$ for notational consistency, since we will later relate a conjecture about the compact Fourier transform for $G$ with Conjecture \ref{c:elliptic} for Levi subgroups $M$ of $G$.)

\begin{rem}\label{r:conj-K}
With notation as in Conjecture \ref{c:elliptic}, by the definition of $\FT_{\cpt, \un}$, for each maximal compact subgroup $K$ of $M$ 
we obtain a diagram 
 \begin{equation}\label{e:diagram-K}
 \xymatrix@+1pc{
    {\Ress(M)_{\Psi(M)}} \ar[r]^{\FT^\vee_{\ess}} \ar[d]_{\res_{K}}
    & {\Ress(M)_{\Psi(M)}}\ar[d]^{\res_{K}}\\
    {  \bigoplus_{H \in \InnT(\overline{K})} R_\un(H)} \ar[r]_{\mathrm{FT_{K}}}
    & {  \bigoplus_{H \in \InnT(\overline{K})} R_\un(H)} }
\end{equation} 
Conjecture \ref{c:elliptic} holds if and only if this diagram commutes up to roots of unity for each $K \in \max(M)$.
\end{rem}

\begin{ex}\label{r:torus-ess}
If $M$ is a torus, then there is a unique unipotent element $u = 1 \in M^\vee$, and every pair in $M^\vee$ is essentially elliptic. The space $\Ress(M)_{\Psi(M)}$ is one dimensional (spanned by $\Pi^M(1, 1, 1)$), the involution $\FT^\vee_{\ess}$ is trivial, and Conjecture \ref{c:elliptic} holds.
\end{ex}

\subsection{Regular unipotent elements}\label{s:reg-uni}

We consider the conjecture when $u \in G^\vee$ is a regular unipotent element. If $u$ is regular unipotent, then $\Gamma_u = \rZ_{G^\vee}$ and every pair $(s, h) \in \Gamma_u^2$ is essentially elliptic (cf. Example \ref{ex-pairs}). For any $s \in \rZ_{G^\vee}$, the component group $A_{\Gamma_u}(s)$ is equal to $\rZ_{G^\vee}/\rZ_{G^\vee}^\circ$. Let $\bar{s}$ be the image of $s$ in this component group, and let $\chi_{\bar{s}}$ be the corresponding weakly unramified character. In $\Ress(G)_{\Psi(G)}$, we have
\begin{equation*}
\Pi(u, s, h) = \sum_{\rho \in \widehat{A}_{\Gamma_u}(s)} \rho(h)\pi(u, 1, \rho) \otimes \chi_{\bar{s}}= \sum_{\rho \in \widehat{A}_{\Gamma_u}(s)} \rho(h)(\mathsf{St}_{G_\rho} \otimes \chi_{\bar{s}}),
\end{equation*}
where $G_\rho$ is the pure inner twist corresponding to $\rho$, and $\mathsf{St}_{G_\rho}$ is its Steinberg representation. 
We then have
\[\res_{\cpt, \un} (\Pi(u,h,s))= \FT_{\cpt,\un}\circ \res_{\cpt, \un}(\Pi(u,s,h)). \]
by the same argument as in \cite[Section 9.2]{ACR}.
This analysis is enough to prove Conjecture \ref{c:elliptic} for $\GL_n$, as we see in the following example.

\begin{ex}\label{s:GLn-ess}
We describe the essentially elliptic case for $\GL_n$. Let $u \in \GL_n(\bC)$ be a unipotent element corresponding to the partition $(\underbrace{a_1,\dots,a_1}_{r_1},\underbrace{a_2,\dots,a_2}_{r_2},\dots,\underbrace{a_\ell,\dots,a_\ell}_{r_\ell})$. Then $\Gamma_u$ contains essentially elliptic pairs if and only if all $\ell=1$ and $r_1=1$, that is, $u$ is a regular unipotent element. Indeed, we have $\Gamma_u \simeq \prod_{i = 1}^\ell \GL_{r_i}(\bC)$. If $s, h \in \Gamma_u$ are semisimple commuting elements, then $s, h$ are contained in some maximal torus $S_u\supset \rZ_{G^\vee}$ of $\Gamma_u$. Since $S_u$ is contained in $Z_{\Gamma_u}(s, h)$, if $Z_{\Gamma_u}(s, h)/\rZ_{G^\vee}$ is finite then $S_u=\rZ_{G^\vee}$. In this case, $\ell=1$ and $r_1=1$ as claimed. The converse is immediate. Thus Conjecture \ref{c:elliptic} holds when $\mathbf{M} = \GL_n$. We note that in this case
 \[\Ress(\GL_n(F))_{\Psi(\GL_n(F))}=\mathbb C\cdot \mathsf{St}_{\GL_n(F)}\] is one dimensional and $\FT^\vee_\ess$ is the identity. 
\end{ex}

\section{The compact Fourier transform}\label{s:rigid} We now extend the conjecture from the elliptic case. 
Let $\cH(G)=C_c^\infty(G)$ be the Hecke algebra of $G$ of compactly supported locally constant functions. If $K$ is any compact open subgroup of $G$, we may embed $\cH(K)=C_c^\infty(K)\hookrightarrow \cH(G)$ by extending the functions by zero on $G\setminus K$. Let $\cH(G)_c$ be the subspace of $\cH(G)$ spanned by the images of these embeddings as $K$ ranges over the set of maximal compact open subgroups of $G$. Let $\overline \cH(G)$ be the cocenter of $\cH(G)$, that is the quotient of $\cH(G)$ by the subspace spanned by the functions $f-~^gf$, where $f\in \cH(G)$, $g\in G$. It is known that this is the same as $\cH(G)/[\cH(G),\cH(G)]$. Let $\overline\cH(G)_c$ denote the image of $\cH(G)_c$ in $\overline\cH(G).$ Also set $\overline \cH(K)=\cH(K)/[\cH(K),\cH(K)]$.

If $K$ is a maximal compact open subgrop of $G$, the linear map $\cH(K)\hookrightarrow \cH(G)_c\twoheadrightarrow \overline\cH(G)_c$ factors through $i_K:\overline\cH(K)\to \overline \cH(G)_c$. It is clear that the resulting map 
\[
i:\bigoplus_K \overline\cH(K)\to \overline \cH(G)_c,\quad i=\sum_K i_K,
\]
where $K$ ranges over the (finite) set of representatives of the conjugacy classes of maximal compact open subgroups of $G$, is surjective. Since $K$ is a compact group, we may identify, if convenient, $R(K)\cong \overline\cH(K)$. 

If $\pi$ is any smooth admissible representation of $G$, let $\Theta_\pi:\cH(G)\to\bC$ be its character. Since the character if $G$-invariant, we may regard $\Theta_\pi$ as a linear functional $\Theta_\pi:\overline\cH(G)\to\bC$. Denote by $\Theta^c_\pi$ its restriction to $\overline\cH(G)_c$.

The same definitions apply to every inner twist $G'$ of $G$. If $D$ is a finite $\bZ$-linear combination of irreducible representations of the various inner twists $G'$, denote by $\Theta_D$ the corresponding linear combination of characters and by $D_c$ the restriction to $\oplus_{G' \in \InnT^p(G)} \overline\cH(G)_c$ as above.

\subsection{The restriction to compact elements} Recall the spaces $\CC[\cY(\Gamma_u)/_\sim]^{\Gamma_u}$ defined in Section \ref{s:ell-pairs}, particularly Corollary \ref{cor-cpt-ell}. For every compact pair $(s,h)\in \cY(\Gamma_u)/_\sim$, define $\Pi(u,s,h)$ as in (\ref{eqn:pi_ush}), and let $\Pi(u,s,h)_c$ be the restriction to $\oplus_{G' \in \InnT^p(G)} \overline\cH(G)_c$. First we note that $\Pi(u, s, h)_c$ does not depend on choice of representative for the $\sim$-equivalence class of $(s, h)$.

\begin{prop}
Suppose $(s, h), (s', h') \in \cY(\Gamma_u)$ with $(s, h) \sim (s', h')$. Then $\Pi(u, s, h)_c = \Pi(u, s', h')_c$.
\end{prop}

\begin{proof}
The relation $\sim$ is generated $\sim_R$ and $\sim_L$. As discussed below (\ref{eqn:pi_ush}), if $(s, h) \sim_L (s, h')$, then $\Pi(u, s, h) = \Pi(u, s, h')$. Now assume $(s, h) \sim_R (s', h)$. We have $\gamma s \gamma^{-1} = s'z$ for some $\gamma \in \rZ_{\Gamma_u}(h)$ and $z \in S^\vee$ a maximal torus in $\rZ_{\Gamma_u}(s, h)$. Let $M^\vee = \rZ_{G^\vee}(S^\vee)$, let $M$ be the corresponding Levi of $G$, and choose a parabolic of $G$ with Levi component $M$.
By Lemma \ref{l:ind-Pi}, we have $\Pi(u, s, h) = i_P^G \Pi^M(u, s, h)$, so it suffices to show $\Pi^M(u, s, h)_c = \Pi^M(u, s', h)_c$. Let $\varphi \in \hat{A}_{\Gamma_u^{M^\vee}(s)}$. Since $z \in \rZ_{M^\vee}^\circ$, by 
 Lemma \ref{l:center-twist} we have $\pi^{M_\varphi}(u, s, \varphi) = \pi^{M_\varphi}(u, s', \varphi^\gamma) \otimes \chi$ for some unramified character $\chi$ of $M_\varphi$. 
 The result follows. 
\end{proof}

\subsection{The compact representation space}\label{s:comact-space}

Let 
 \begin{equation}
 \Rcpt(G)=\text{span}\langle \Pi(u,s,h)_c\mid u\in G^\vee_\un/G^\vee,\ (s,h)\in \cY(\Gamma_u)/_\sim\rangle.
 \end{equation}

Note that since every $\sim$-equivalence class in $\cY(\Gamma_u)$ has a representative $(s_0, h_0)$ with $s_0$ and $h_0$ of finite order, we may choose a lift $\Pi(u, s_0, h_0)$ of each $\Pi(u, s, h)_c$ such that  
 all of the irreducible representations
 that occur in the virtual linear combination $ \Pi(u,s_0,h_0)$ are tempered. 

 To understand this space, we recall the notion of rigid/compact representations from \cite{CH2} and the connection with the theory of Hopf systems defined in \cite{BDK} and refined in \cite{Da}. In {\it loc. cit.}, the results are stated for the full categories of representations of the $p$-adic group, but since the category of unipotent representations of $G$ is a direct summand of the full category of smooth representations of $G$, we state the results directly in the unipotent setting. Let $G'$ be a pure inner twist of $G$. Following \cite[\S6.7]{CH2}, see also \cite[\S3]{Ci-inv}, define the compact quotient 
  \begin{equation}
  \overline R(G')_{\un,c}= R_\un(G')/ R(G')_{\un,\mathrm{diff}};
  \end{equation}
here $R(G')_{\un,\mathrm{diff}}\subset  R_\un(G')$ is the $\CC$-span of $i_{M'}^{G'}(\sigma)-i_{M'}^{G'}(\sigma\otimes\chi)$, where $M'$ ranges over the set of standard Levi subgroups of $G'$, $\sigma\in R_\un(M')$, and $\chi$ is an unramified character of $M'$.

We now work with the relevant spaces for $G$; the same discussion applies to every inner twist $G'$. 
Let $\overline \cH(G)_{\un,c}$ denote the (unipotent) compact cocenter, which is the subspace of $\overline\cH(G)_c$ generated by the image of the unipotent representations. The final space of interest is the Grothendieck group $K(G)_{\un}$ of finitely generated (unipotent) projective $G$-representations. There are the following known relations between them:
\begin{enumerate}
    \item (Abstract Selberg principle). The Hattori--Stallings rank map $\mathsf{rk}_{HS}: K(G)_{\un}\to \overline \cH(G)_{\un,c}$ is a linear isomorphism. See \cite[Theorem 1.6]{Da}, also \cite[\S7]{CH2}.
    \item (Rigid/Compact trace Paley--Wiener Theorem). The trace map
\begin{equation}\label{e:tr-c}
\tr_c: \overline\cH(G)_{\un,c}\longrightarrow \overline R(G)_{\un,c}^*, \ [f]\mapsto (\pi\mapsto \tr(f, \pi)).
\end{equation}
is surjective. See \cite[Proposition 6.10]{CH2} (and \cite[Corollary 3.7]{Ci-inv}). Kazhdan's Density Theorem \cite{Kaz} can be used to conclude that $\tr_c$ is injective. This gives a perfect duality between $\overline R(G)_{\un,c}$ and $\overline \cH(G)_{\un,c}$ via $\tr_c$.
\end{enumerate}

Let $\mathcal A(G)$ denote any one of $\overline R(G)_{\un,c}$, $\overline \cH(G)_{\un,c}$, or $K(G)_{\un}$. Together with the parabolic induction and restriction maps, $\mathcal A$ is a Hopf system on $G$ in the sense of \cite[Definition 2.2]{Da}. If $M<N$ are standard Levi subgroups of $G$, let $i_M^N$ and $r_N^M$ denote the linear maps given by induction and restriction, respectively, for $\mathcal A$. They are compatible via the $\mathsf{rk}_{HS}$ and dual via $\tr_c$, as expected.

For a standard Levi subgroup $M$, let $Z_M$ be the center of $M$, and let $d(M)=\dim Z_M$. Let $N_M=\{w\in W_M\backslash W_G/W_M\mid w W_M w^{-1}=W_M\}.$ Define two filtrations on $\mathcal A(G)$:
\begin{enumerate}
    \item $\mathcal A_i(G)=\sum_{d(M)>i}\im i_M^G$;
    \item $\mathcal A^i(G)=\cap_{d(M)>i}\ker r_G^M$.
\end{enumerate}
Notice that $\{\mathcal A_i(G)\}$ is decreasing and $\{\mathcal A^i(G)\}$ is increasing. For each $d$, choose an ordering of the set of standard Levi subgroups $M$ with $d(M)=d$ and define, with respect to this ordering, the BDK-operators
\[
T_d=\prod_{M, d(M)=d} (i_M^G\circ r_G^M-|N_M|).
\]
and
\[
A_d=T_{d_0}\circ T_{d_0-1}\circ\dots\circ T_d\text { and } A^d=T_{d}\circ T_{d+1}\circ\dots\circ T_{d_0},
\]
where $d_0$ is the (split) rank of $G$. 
Set $A_I=A_{d(G)}$ and $A^R=A^{d(G)}$. The behaviors of $A_d$ and $A^d$ are dictated by the following two properties that are known to hold for our three examples of $\mathcal A$ by \cite[Corollaries 4.21, 4.23]{Da}:
\begin{enumerate}
    \item[(A-ind)] $i_{w(M)}^G\circ w=i_M^G$, $w\in W_G$;
    \item[(A-res)] $r_G^{w(M)}=w\circ r_G^M$, $w\in W_G$.
\end{enumerate}
Hence, by \cite[Proposition 2.5]{Da}, for every $d$,
\begin{equation}
    \ker A_d=\mathcal A_d(G),\quad \im A^d=\mathcal A^d(G), \quad A_d=A^d\text{ and } \mathcal A(G)=\mathcal A_d(G)\oplus \mathcal A^d(G).
\end{equation}
One can also consider the associated graded objects $\mathsf{gr}_*\mathcal A(G)$ and $\mathsf{gr}^*\mathcal A(G)$ with the associated induction/restriction maps; these are also Hopf systems. 

Let $\cL$ be the set of standard Levi subgroups of $G$. We have the following theorem.

\begin{thm}[{\cite[Theorem 4.25]{Da}}]\label{t:dat} There is an isomorphism of Hopf systems between $\mathcal A(G)$ and the trivial Hopf system $\mathfrak H(M\mapsto \overline{R}_\un(M)_{\Psi(M)})$ (see \cite[\S2.7]{Da}) . In particular, 
    \[
    \mathcal A(G)\cong \mathsf{gr}_*\mathcal A(G)\cong \mathsf{gr}^*\mathcal A(G)\cong \left(\bigoplus_{M\in \cL} \overline R_{\un}(M)_{\Psi(M)} \right)/ \sim_{G}.
    \]
    Here $\sim_G$ is the equivalence relation given by conjugation.
\end{thm}

    We are primarily interested in the following statement. 
    \begin{cor}
    \begin{equation}\label{e:gr}
    \mathsf{gr}_*\overline R(G)_{\un,c}\cong \left(\bigoplus_{M\in \cL} \overline R_{\un}(M)_{\Psi(M)} \right)/ \sim_{G}.
    \end{equation}
    \end{cor}
   \begin{proof}
        This is a particular instance of the stronger result in Theorem \ref{t:dat}. 
        
        It can also be obtained more easily from Dat's results, as follows. Firstly, notice that $\overline R(G)_{\un,c}$ satisfies property (A-ind); this is immediate from the well-known fact that $R(G)$ has this property \cite[Lemma 5.4(iii)]{BDK}. Then \cite[Proposition 2.10(i)]{Da}, which is a formal calculation, already shows that $\mathsf{gr}_*\overline R(G)_{\un,c}$ is isomorphic to the trivial Hopf system $\mathfrak H(M\mapsto \mathcal A(M)/\mathcal A_I(M))$, with $\mathcal A=\overline R(G)_{\un,c}$.
        
          Finally, to obtain (\ref{e:gr}), we just need to note that 
          \[\overline R(G)_{\un,c}/(\overline R(G)_{\un,c})_I\cong  \overline R_{\un}(G)_{\Psi(G)}.\] 
         \end{proof}

    Moreover, this implies that the induction map 
    \[
    \oplus_M i_{M}^{G}: \left(\bigoplus_{M\in \cL} \overline R_{\un}(M)_{\Psi(M)} \right)/ \sim_{G}\longrightarrow \overline R(G)_{\un,c}
    \]
    gives a section of this isomorphism.
We apply this to all pure inner twists of $G$ and use Theorem \ref{t:ess}. Set $\overline R(G)^p_{\un,c}=\oplus_{G' \in \InnT^p(G)} \overline R(G')_{\un,c}$. This leads to the linear isomorphism  $\ind_c=\oplus_{M} i_{M}^{G}$ 
(cf. (\ref{r:par-ind-twists}))
\begin{equation}\label{e:ind-c}
    \ind_c: \bigoplus_{M\in \cL} \Ell_{\un}^p(M)_{\Psi(M)} / \sim_{G}\longrightarrow \overline R(G)^p_{\un,c}.
\end{equation}

\begin{prop}\label{p:cpt-space}
There exist natural isomorphisms 
\[\overline R(G)^p_{\un,c}\cong \Rcpt(G)\cong  \bigoplus_{M\in \cL}  \Ell(M^\vee)/(\rZ_{M^\vee}^\circ \times \rZ_{M^\vee}^\circ)/\sim_{G^\vee}\cong\bigoplus_{u\in G^\vee_\un/G^\vee}\CC[\cY(\Gamma_u)/_\sim]^{\Gamma_u}.\]
\end{prop}

\begin{proof}
  By Corollary \ref{c:ess-ell} and Corollary \ref{cor-cpt-ell}, a basis of the image of $\ind_c$ in (\ref{e:ind-c}) is given by the set $\{\Pi(u,s,h)\mid u\in G^\vee_\un/G^\vee,\ (s,h)\in \cY(\Gamma_u)/_\sim\}$ mod $R(G')_{\un,\mathsf{diff}}$. By the duality between $\overline R(G)^p_{\un,c}$ and $\overline H(G)_{\un}^p$ given by the compact trace Paley--Wiener theorem (\ref{e:tr-c}), it follows that $\{\Pi(u,s,h)_c\mid u\in G^\vee_\un/G^\vee,\ (s,h)\in \cY(\Gamma_u)/_\sim\}$ is also linearly independent and the chain of isomorphisms in the proposition follows. 
\end{proof}

We remark that the inverse 
\begin{equation}\label{e:r-comp}
  \Rcpt(G)\longrightarrow \bigoplus_{M\in \cL} \Ell^p_{\un,\ess}(M)_{\Psi(M)} / \sim_G.
  \end{equation}
of $\ind_c$ may be defined as follows: given a triple $(u, s, h)$ with $\kappa(u, s, h) = (M^\vee, u, s, h)$, let $M$ be a Levi subgroup of $G$ corresponding to $M^\vee$. We then define a map by sending $\Pi(u, s, h)_c$ to $\Pi^M(u, s, h)$ and extending linearly.

 \subsection{Compact Fourier transform}

 \begin{defn}
 The \emph {dual compact (nonabelian) Fourier transform} is the linear map given by
 \begin{equation}
 \FT^\vee_{\cpt}\colon  \Rcpt(G)\to  \Rcpt(G),\quad \FT^\vee_{\cpt}(\Pi(u,s,h)_c)=\Pi(u,h,s)_c.
 \end{equation}
 \end{defn}
Then we may expect the following relation, just like in the case of elliptic representations \cite{ACR}.

  \begin{conj}\label{c:compact} Let $G$ be a connected reductive $F$-split group. The following diagram commutes up to roots of unity:
 \begin{displaymath}\label{e:diagram-cpt}
  \xymatrix@+1pc{
    {\Rcpt(G)} \ar[r]^{\mathrm{FT}^\vee_{\cpt}} \ar[d]_{\res_{\cpt,\un}}
    & {\Rcpt(G)}\ar[d]^{\res_{\cpt,\un}}\\
    {  \cC(G)_{\cpt,\un}} \ar[r]_{\mathrm{FT_{\cpt,\un}}}
    & {  \cC(G)_{\cpt,\un}} }
\end{displaymath}
In other words, for each  unipotent element $u\in G^\vee$, compact pair $(s,h) \in \cY(\Gamma_u)/_\sim$, and maximal compact open subgroup $K$ of $G$, there exists a root of unity $\zeta=\zeta(u,s,h,K)$ such that
\[\res_K (\Pi(u,h,s)_c)=\zeta \cdot (\FT_{\cpt,\un}\circ \res_K)(\Pi(u,s,h)_c).
 \]
 \end{conj}

 \begin{rem}\label{r:conj-cpt}
 Analogous to Conjecture \ref{c:elliptic}, for each maximal compact subgroup $K$ of $G$, we have a corresponding diagram
 \begin{equation}\label{e:diagram-K-cpt}
 \xymatrix@+1pc{
    {\Rcpt(G)} \ar[r]^{\FT^\vee_{\cpt}} \ar[d]_{\res^G_{K}}
    & {\Rcpt(G)}\ar[d]^{\res^G_{K}}\\
    {  \bigoplus_{H \in \InnT(\overline{K})} R_\un(H)} \ar[r]_{\mathrm{FT_{K}}}
    & {  \bigoplus_{H \in \InnT(\overline{K})} R_\un(H)} }
\end{equation} 
 and Conjecture \ref{c:compact} holds if and only if this diagram commutes up to roots of unity for each $K$ (cf. Remark \ref{r:conj-K}).
 \end{rem}

\begin{ex}\label{ex:reg-cpt}
Again, we consider the case when $u = u_{\text{reg}}$ is regular unipotent. In this case $\cY(\Gamma_u)/_\sim = \cY(\rZ_{G^\vee})/(\rZ_{G^\vee}^\circ \times \rZ_{G^\vee}^\circ)$, which is in bijection with the set of pairs in the component group $\rZ_{G^\vee}/\rZ_{G^\vee}^\circ$. The conjecture reduces to Conjecture \ref{c:elliptic}, and by Section \ref{s:reg-uni} we have
\begin{equation*}
\res_K(\Pi(u_{\text{reg}}, h, s)_c) = \FT^\vee_\cpt \circ \res_K (\Pi(u_{\text{reg}}, s, h)_c)
\end{equation*}
for all compact pairs $(s, h)$ in $\cY(\rZ_{G^\vee})$ and all compact open subgroups $K \subset G$.
\end{ex}

 The main goal for the rest of the paper is to show the relation between Conjectures \ref{c:compact} and \ref{c:elliptic}, and in particular to show that
Conjecture \ref{c:compact} follows from Conjecture \ref{c:elliptic} in certain cases.

\begin{ex}\label{ex:111}
Consider the triple $(u, s, h) = (1, 1, 1)$. If we let $T$ be a split maximal torus in $G$, then $(1, 1)$ is essentially elliptic in $T^\vee$, and $\Pi(1, 1, 1) = i_T^G \Pi^T(1, 1, 1) = i_T^G 1_T$. Let $K$ be a maximal compact subgroup of $G$ containing the Iwahori $I$.
By Proposition \ref{p:par-ind}, we have
\begin{equation*}
\res_K \Pi(1, 1, 1) = \sum_g i_{\overline{K \cap T^g}}^{\overline{K}} \res_{K \cap T^g}^{T^g} (1_T)^g.
\end{equation*}
Since $T^g = T$ for all $g \in N_G(T)$ and $\overline{K \cap T}$ is a maximal torus in $\overline{K}$, we have
\begin{equation*}
\res_K \Pi(1, 1, 1) = \sum_g i_{\overline{B}}^{\overline{K}} 1
\end{equation*}
where $\overline{B}$ is a Borel subgroup of $(\overline{K})^\circ$. So for the triple $(1, 1, 1)$, proving Diagram \ref{e:diagram-K-cpt} commutes is equivalent to showing that $i_{\overline{B}}^{\overline{K}} 1 = \Ind_{(\overline{K})^\circ}^{\overline{K}} (i_{\overline{B}}^{(\overline{K})^\circ} 1)$ is fixed by $\FT_K$. 
Note that when $\overline{K}$ is connected, this is simply saying that $i_{\overline{B}}^{\overline{K}} 1=R_{\overline{K\cap T}}^{\overline K}(1)$ (here the notation $R$ is for the Deligne--Lusztig induction), where $\overline{K\cap T}$ is the maximal split $\bF_q$-torus is fixed by Lusztig's Fourier transform. This is well known, see for example \cite[p. 383--384]{Car}. In the non-connected case, the statement is again equivalent to the fact that the non-connected Fourier transform fixes $R_{\overline{K\cap T}}^{\overline K}(1)$, where $R$ is the non-connected extension of the Deligne--Lusztig induction as in \cite[Proposition 1.3]{DM}. This follows for example from \cite[Theorem 5.8 and Proposition 6.4]{DM} where the relation between the characters of the Weyl group of $\overline K$ and the almost-characters is shown to hold just like in the connected case.

\end{ex}

\subsection{The nonabelian Fourier transform and parabolic induction}\label{s:FT-ind}

\begin{hyp}\label{h:FT-comm}
    Suppose $\mathbf{H}$ is a split connected reductive group over $\mathbb F_q$, and let $H = \mathbf{H}(\mathbb{F}_q)$ 
    Let $R_\un(H)$ denote the $\bC$-span of irreducible unipotent $H$-characters and let $\FT_H:R_\un(H)\to R_\un(H)$ be the nonabelian Fourier transform. 

    If $P=MN$ is an $\mathbb{F}_q$-parabolic subgroup of $H$ with Levi subgroup $M$, then
    \[
    \FT_H\circ i_{P}^{H}=i_{P}^{H}\circ \FT_M.
    \]
\end{hyp}

\begin{rem}
Hypothesis \ref{h:FT-comm} is known to hold when $H$ is a classical group, see  \cite[\S6.4]{MW}. 
\end{rem}

\begin{lem}\label{l:ind-res-hyp}
With notation as in Hypothesis \ref{h:FT-comm}, let $r^{M}_{H}: R_\un(H) \to R_\un(M)$ denote the Harish-Chandra restriction functor. Then Hypothesis \ref{h:FT-comm} holds if and only if 
\begin{equation*}
\FT_{M} \circ r^{M}_{H} = r^{M}_{H} \circ \FT_H
\end{equation*}
for all $\mathbb{F}_q$-Levi subgroups $M$.
\end{lem}

\begin{proof}
Let $\la \cdot, \cdot \ra_M$ be the character pairing for $M$ (and similarly for $H$), and assume Hypothesis \ref{h:FT-comm}. Given $\pi \in R_\un(H)$, one checks that $\la \FT_{M} \circ r^{M}_{H}(\pi), \sigma \ra_M = \la r^{M}_{H} \circ \FT_H(\pi), \sigma \ra_M$ for all $\sigma \in R_\un(M)$. This follows easily from the following two facts: the first is that $\la r_{H}^{M}\pi, \sigma\ra_M = \la \pi, i_{P}^{H} \sigma \ra_H$ for all $\pi \in R_\un(H), \sigma \in R_\un(M)$;
the second is that $\la \FT_H\pi, \pi'\ra_H = \overline{\la \pi, \FT_H\pi' \ra}_H$ for all $\pi, \pi' \in R_\un(H)$ (where the bar indicates complex conjugation). 
To prove the second fact, suppose $\pi$ and $\pi'$ are irreducible. Then $\la \FT_H\pi, \pi' \ra_H$ and $\la \pi, \FT_H\pi' \ra_H$ are both zero if $\pi$ and $\pi'$ are not in the same family. If they are in the same family, then the fact follows easily from the definition of $\FT_H$.
\end{proof}

\begin{lem}\label{l:FT-res-ind}
    Suppose $K$ is a parahoric subgroup of $G$ containing $I$ and that $P = MN$ is a standard parabolic subgroup of $G$. Let $V=i_P^G(\pi_0)$ be a parabolically induced representation. Assume Hypothesis \ref{h:FT-comm} holds for $\overline K$. Then 
    \[
    \FT_{\overline K}\circ\res_K^G(V)=\bigoplus_{g\in ^K\!W^P} i_{\overline {K\cap P^g}}^{\overline K}(\FT_{\overline{K\cap M^g}}\circ\res^{M^g}_{K\cap M^g}(\pi_0^g)).
    \]
\end{lem}

\begin{proof}
    Immediate from (\ref{e:mackey}) and Hypothesis \ref{h:FT-comm}. 
\end{proof}

\begin{lem}\label{l:non-max-K}\footnote{We thank J.-L. Waldspurger for pointing this out to us, and for sketching the proof.}
Assume Diagram \ref{e:diagram-K} commutes (up to roots of unity) for some maximal parahoric $K$ and that Hypothesis \ref{h:FT-comm} holds for $\overline{K}$. Then for any parahoric subgroup $K_0 \subset K$, the corresponding diagram
 \begin{displaymath}
 \xymatrix@+1pc{
    {\Ress(G)_{\Psi(G)}} \ar[r]^{\FT^{\vee}_{\ess}} \ar[d]_{\res^G_{K_0}}
    & {\Ress(G)_{\Psi(G)}}\ar[d]^{\res^G_{K_0}}\\
    {   R_\un(\overline{K}_0)} \ar[r]_{\mathrm{FT_{\overline K_0}}}
    & {  R_\un(\overline{K}_0)} }
\end{displaymath} 
also commutes (up to roots of unity). Here $\res^G_{K_0}$ is parahoric restriction with respect to $K_0$, and $\FT_{\overline K_0}$ is Lusztig's nonabelian Fourier transform for $\overline{K}_0$. 
\end{lem}

\begin{proof}
Fix a triple $(u, s, h)$ such that $(s, h) \in \cY(\Gamma_u)_{G^\vee-\ess}$. Given a parahoric $K_0 \subset K$, we have that $\overline{K}_0 = K_0/U_{K_0}$ is a Levi subgroup of $\overline{K}$ and $\res^G_{K_0} = r_{\overline K}^{\overline K_0} \circ \res_K^G$. 
By assumption, there exists a root of unity $\zeta$ such that $\res_K^G \circ \FT_\ess^\vee(\Pi(u, s, h)) = \zeta\FT_{\overline{K}} \circ \res_K^G(\Pi(u, s, h))$. We have
\begin{align*}
\res_{K_0}^G \circ \FT_\ess^\vee (\Pi(u, s, h)) &= r_{\overline K}^{\overline K_0} \circ \res_K^G \circ \FT_\ess^\vee (\Pi(u, s, h))
= \zeta r_{\overline K}^{\overline K_0} \circ \FT_{\overline K} \circ \res_K^G(\Pi(u, s, h))\\
&= \zeta \FT_{\overline K_0} \circ r_{\overline K}^{\overline K_0} \circ \res_K^G(\Pi(u, s, h)= \zeta \FT_{\overline K_0} \circ \res_{K_0}^G(\Pi(u, s, h)),
\end{align*}
where we are using the form of Hypothesis \ref{h:FT-comm} in Lemma \ref{l:ind-res-hyp}. 
\end{proof}

To prove the next theorem, we must make an additional hypothesis related to functoriality in the local Langlands correspondence for Levi subgroups. This is a version of ``weak invariance" of the LLC (see \cite[Conjecture 5.2.7]{Hai14}), and it holds in the case when $G = \GL_n$ \cite[Proposition 5.2.6]{Hai14}. 

To explain the hypothesis, let $M$ be a standard $F$-Levi subgroup of $G$, and let $g \in N_G(T)$. Let $c_g: M \to M^g$ be the isomorphism given by conjugation by $g$. Thinking of $M^\vee$ and $(M^g)^\vee$ as subgroups of $G^\vee$, the map $c_g$ induces an isomorphism of dual groups $(M^g)^\vee \to M^\vee$, given by conjugation by an element $g^\vee \in G^\vee$ (cf. \cite[2.1]{Bor}). Let $g' = (g^\vee)^{-1}$, and given $x \in G^\vee$, write $x^{g'}$ for $g'x(g')^{-1}$. The hypothesis is then as follows.

\begin{hyp}\label{h:FT-twist}
With notation as above, for any $(\varphi_{(u, s)}, \rho) \in \Phi^p_{e, \un}(M^\vee)$ such that $M_\rho = M$, we have
\begin{equation*}
(\pi^M(u, s, \rho))^g = \pi^{M^g}(u^{g'}, s^{g'}, \rho^{g'}).
\end{equation*}
\end{hyp}
Note that if $g \in N_G(M)$, then $g$ induces an automorphism of the root datum of $M$, and hence a $W_F$-equivariant automorphism of the Dynkin diagram of $M$. In this case Hypothesis \ref{h:FT-twist} is proved in \cite[Theorem~4.1(iii)]{So1}.

\begin{thm}\label{t:ell-to-com-K}
Assume Hypotheses \ref{h:FT-comm} and \ref{h:FT-twist}, and that Diagram \ref{e:diagram-K} commutes for every Levi subgroup $M$ of $G$. Suppose $K$ is a maximal parahoric subgroup of $G$ containing $I$. Then Diagram \ref{e:diagram-K-cpt} commutes for $K$.
\end{thm}

\begin{proof}
Let $u$ be a unipotent element of $G^\vee$, and let $(s, h) \in \cY(\Gamma_u)$ be a compact pair. Let $\Pi^M(u, s, h)$ be (a $\sim_G$-representative of) the image of $\Pi(u, s, h)_c$ under the isomorphism (\ref{e:r-comp}). Since the right-hand side of (\ref{e:r-comp}) is a sum over $G$-conjugacy classes of Levis, we may assume $M$ is a standard Levi, and thus we are in the setting of Lemma \ref{l:levi-compact}. For simplicity of notation, let $E$ be the set of all $\rho \in \widehat{A}_{\Gamma_u^{M^\vee}}(s)$ that are $M$-relevant.

 Let $P = MN$ be a parabolic with Levi $M$. 
By Lemma \ref{l:ind-Pi}, we have that $\Pi(u, s, h) = i_P^G \Pi^M(u, s, h)$. 
Since $K$ is parahoric, we have that $\overline K$ is connected and thus has no inner twists. Thus we have that $\res_K^G = \res_K^G \circ p_M$ where $p_M$ is the natural projection $\oplus_{M' \in \InnT^p(M)} R_{\un}(M') \to R_{\un}(M)$. Using this, we have
\begin{align*}
\FT_{\overline K} \circ \res_K^G(\Pi(u, s, h)) &= \FT_{\overline K} \circ \res_K^G i_P^G \Pi^M(u, s, h)\\
&= \sum_{\rho \in E} \rho(h) \FT_{\overline K} \circ \res_K^G i_P^G \pi^M(u, s, \rho),
\end{align*}

By Lemma \ref{l:FT-res-ind} this sum is equal to 
\begin{equation*}
\sum_{\rho \in E} \rho(h)(\sum_{g\in ^K\!W^P} i_{\overline{K \cap P^g}}^{\overline K} (\FT_{\overline{K \cap M^g}} \circ \res_{K \cap M^g}^{M^g} (\pi^M(u, s, \rho))^g)),
\end{equation*}
which is equal to 
\begin{equation}\label{e:big-sum}
\sum_{g\in ^K\!W^P} i_{\overline{K \cap P^g}}^{\overline K} (\FT_{\overline{K \cap M^g}} \circ \res_{K \cap M^g}^{M^g} \sum_{\rho \in E} \rho(h)(\pi^M(u, s, \rho))^g).
\end{equation}
By Hypothesis \ref{h:FT-twist}, we have 
\begin{align*}
\res_{K \cap M^g}^{M^g} \sum_{\rho \in E} \rho(h)\pi^M(u, s, \rho))^g &= \res_{K \cap M^g}^{M^g} \sum_{\rho \in E} \rho(h)\pi^{M^g}(u^{g'}, s^{g'}, \rho^{g'})\\
&= \res_{K \cap M^g}^{M^g} \sum_{\rho \in E}  \rho^{g'}(h^{g'}) \pi^{M^g}(u^{g'}, s^{g'}, \rho^{g'})\\
&= \res_{K \cap M^g}^{M^g} \Pi^{M^g}(u^{g'}, s^{g'}, h^{g'}).
\end{align*}
Here we are using that $K \cap M^g$ is a parahoric in $M^g$, so again $\res_{K \cap M^g}^{M^g}$ is zero on $R_\un(M')$ for all $M' \in \InnT^p(M^g)$ with $M' \neq M^g$. Now
Lemma \ref{l:non-max-K} implies that 
\begin{align*}
\FT_{\overline{K \cap M^g}} \circ \res_{K \cap M^g}^{M^g} \Pi^{M^g}(u^{g'}, s^{g'}, h^{g'}) &= \res_{K \cap M^g}^{M^g} \FT_{\ess}^\vee \Pi^{M^g}(u^{g'}, s^{g'}, h^{g'})\\
 &= \res_{K \cap M^g}^{M^g} \Pi^{M^g}(u^{g'}, h^{g'}, s^{g'}),
\end{align*}
so (\ref{e:big-sum}) is equal to 
\begin{equation*}
\sum_{g\in ^K\!W^P} i_{\overline{K \cap P^g}}^{\overline K}(\res_{K \cap M^g}^{M^g}  \Pi^{M^g}(u^{g'}, h^{g'}, s^{g'})).
\end{equation*}
Again using Hypothesis \ref{h:FT-twist}, this is equal to 
\begin{equation*}
\sum_{g\in ^K\!W^P} i_{\overline{K \cap P^g}}^{\overline K}(\res_{K \cap M^g}^{M^g} ( \Pi^{M}(u, h, s))^g).
\end{equation*}
Proposition \ref{p:par-ind} tells us that this is equal to 
\begin{equation*}
 \res_K^G i_P^G \Pi^M(u, h, s)
= \res_K^G \Pi^G(u, h, s),
\end{equation*}
as desired. 
\end{proof}

\begin{rem}
Note that the same proof shows that if Diagram \ref{e:diagram-K} commutes up to roots of unity for all Levi subgroups $M$ of $G$, then so does Diagram \ref{e:diagram-K-cpt}, as long as the root of unity $\zeta(u, s, h, K \cap M)$ appearing for $M$ is the same as the root of unity $\zeta(u^{g'}, s^{g'}, h^{g'}, K \cap M^g)$ appearing for every twist $M^g$ of $M$. 
\end{rem}

\begin{cor}\label{c:levis-imply-compact}
Assume $G$ has simply connected derived subgroup, and assume Hypotheses \ref{h:FT-comm} and \ref{h:FT-twist}. If Diagram \ref{e:diagram-K} commutes for every Levi subgroup $M$ of $G$ and every maximal compact subgroup $K$ of $M$, then Conjecture \ref{c:compact} holds for $G$.
\end{cor}

\begin{proof}
This follows immediately from Proposition \ref{t:ell-to-com-K} (cf. Remarks \ref{r:conj-K} and \ref{r:conj-cpt}) since every maximal compact subgroup is a maximal parahoric.
\end{proof}

\begin{cor}\label{c:GL-G2}
Assume Hypotheses \ref{h:FT-comm} and \ref{h:FT-twist}. Then Conjecture \ref{c:compact} holds when $\bG$ is $\GL_n$ or $G_2$.
\end{cor}

\begin{proof}
Assume $G$ is type $G_2$. Conjecture \ref{c:elliptic} has been proven for $G = G_2$ \cite{Ciu}. The other Levi subgroups of $G$ are isomorphic to $\GL_2$ or a torus, so Conjecture \ref{c:compact} follows from Remark \ref{r:torus-ess} and Section \ref{s:GLn-ess} by applying Corollary \ref{c:levis-imply-compact}.

Now assume $G$ is $\GL_n$. Given a Levi subgroup $M^\vee$ of $G^\vee$ and a unipotent element $u \in M^\vee$, similar logic as in Example \ref{s:GLn-ess} shows that $\Gamma_u^{M^\vee}$ contains essentially elliptic pairs if and only if $u$ is regular in $M^\vee$. To give more detail, note that $M^\vee \simeq \prod_{i = 1}^c \GL_{n_i}$ for some integers $n_i$. A unipotent element $u$ corresponds to a $c$-tuple of partitions $(\lam_1, \dots, \lam_c)$, where each $\lam_i$ is a partition of $n_i$. If $\Gamma^{M^\vee}_u$ contains essentially elliptic pairs, then $\rZ_{\Gamma^{M^\vee}_u}/\rZ_{M^\vee}$ must be finite, which occurs if and only if each $\lam_i$ is rectangular, in which case $\Gamma_u \simeq \prod_{i = 1}^c \GL_{r_i}(\bC)$ for some integers $r_i$ with $r_i \vert n_i$. Using the fact that any two commuting semisimple elements $s, h \in \Gamma_u$ lie in a common torus in $\Gamma_u$, we see that any maximal torus in $\Gamma_u$ must be equal to the center of $M^\vee$. This implies that all $r_i = 1$, i.e. that $u$ is regular. 
Thus the essentially elliptic conjecture follows for $M$ by Example \ref{s:GLn-ess}, and then Conjecture \ref{c:compact} follows from Corollary \ref{c:levis-imply-compact}.
\end{proof}

\section{Examples in type A}\label{s:ex-A}

We finish the paper by considering Conjecture \ref{c:compact} in the case when $\bG = \SL_n$ or $\PGL_n$. For $\bG = \SL_n$, we prove the conjecture; for $\bG = \PGL_n$ we classify compact pairs and prove the conjecture in the special case of $n$ prime.  

First we fix the following notation: given a partition $\lam = (a_1, \dots, a_1, a_2, \dots, a_2, \dots, a_\ell, \dots, a_\ell)$ of a positive integer $n$, write $\gcd(\lam)$ for $\gcd\{a_i \mid 1\leq i \leq \ell\}$.

\subsection{$\SL_n(F)$}

In this subsection, we let $\bG = \SL_n$. First we we classify compact pairs for $G^\vee = \PGL_n(\bC)$. Let $\zeta_n = e^{\frac{2\pi i}{n}}$, and let $h_0 = \text{diag}\{\zeta_n^{n - 1}, \zeta_n^{n - 2}, \dots, \zeta_n, 1\} \in G^\vee$. Let  $w_c \in G^\vee$ be the permutation matrix corresponding to the $n$-cycle $(1~ 2 \dots n)$.

\begin{lem}\label{lem-pairs-pgln}
With notation as above, a set of representatives for $\cY(\PGL_n(\bC))/_\sim$ is given by 
\begin{equation*}
\{(h_0, 1), (h_0, w_c), \dots, (h_0, w_c^{n-1})\}.
\end{equation*}
In particular, $\PGL_n(\bC)$ contains $n$ compact pairs.
\end{lem}

\begin{proof}
We have $A_{G^\vee}(h_0) \simeq \bZ/n\bZ$, which we can take to be generated by $w_c$. By Lemma \ref{l:ell-comp}, there are $n$ $\sim_L$-equivalence classes of pairs $(h_0, h)$: these are represented by $(h_0, 1), (h_0, w_c),$ $\dots, (h_0, w_c^{n-1})$.

Now suppose $s \in \PGL_n(\bC)$ is semisimple of finite order. Using \cite[\S2, Example 1]{Re4}, there exists an $n$-tuple of nonnegative integers $(s_0, \dots, s_{n - 1})$ (the Kac coordinates of $s$) such that $s$ is $\PGL_n(\bC)$-conjugate to the diagonal matrix diag $\{\zeta^{-s_0}, \zeta^{-s_0 - s_1}, \dots, \zeta^{-s_0 - \dots - s_{n - 2}}, 1\}$, where $\zeta = e^{\frac{2\pi i}{\sum_{i = 0}^{n - 1} s_i}}$. If $C_{G^\vee}(s)$ is connected, then for any semisimple $h \in C_{G^\vee}(s)$, we have $(s, h) \sim (s, 1) \sim (1, 1)$. If not, then using \cite[Proposition 2.1]{Re4}, there exists a positive integer $k \vert n$ such that the Kac coordinates $(s_0, s_1, \dots, s_{n - 1})$ are of the form $(s_0, \dots, s_{k - 1}, s_0, s_1, \dots, s_{k - 1}, s_0, \dots, s_{k - 1})$ and $A_{G^\vee}(s)$ is generated by $w_c^k$. Thus any compact pair of the form $(s, h)$ for this fixed $s$ is equivalent to one of the form $(s, w_c^{mk})$ for some $m$.

To finish the proof, it suffices to show that for every $m$, we have $(s, w_c^{mk}) \sim_R (h_0, w_c^{mk})$. Let $S \subset \PGL_n(\bC)$ be the subgroup of diagonal matrices of the form $\text{diag }\{t_1, t_2, \dots, t_k, t_1, t_2, \dots, t_k, \dots, t_1, \dots, t_k\}$. Then $S$ is a torus in $\rZ_{\PGL_n(\bC)}(s, w_c^{mk})$, and it is not hard to see that $h_0^{-1}s \in S$. This shows that $(s, w_c^{mk}) \sim (h_0, w_c^{mk})$, and so there are exactly $n$ compact pairs in $\PGL_n(\bC)$. 
\end{proof}

We now prove two general lemmas which will be helpful in the classification of compact pairs in $\Gamma_u$ for a general unipotent element $u$. 

\begin{lem}\label{lem-conn-centralizer}
Let $\Gamma$ be a connected complex reductive group, and let $S \subset \rZ_\Gamma^\circ$ be a torus contained in the center of $\Gamma$. Let $\rho: \Gamma \to \Gamma/S$ be the natural projection map. Then for any semisimple element $s \in \Gamma$, $\rZ_{\Gamma/S}(sS)^\circ = \rho(\rZ_\Gamma(s)^\circ)$. 
\end{lem}

\begin{proof} Let $T \subset \Gamma$ be a maximal torus containing $s$. Note that the root systems for $\Gamma$ and $\Gamma/S$ are the same: identifying these, write $\Phi$ for the roots of $\Gamma$ with respect to $T$. Given $\al \in \Phi$, write $\widetilde{X}_\al$ for the corresponding root group in $\Gamma$ and $X_\al$ for the corresponding root group in $\Gamma/S$. Note that $\rho(\widetilde{X}_\al) = X_\al$ for all $\al \in \Phi$, and $\rho(T) = T/S$. The lemma then follows from the explicit descriptions of $\rZ_\Gamma(s)^\circ$ and $\rZ_{\Gamma/S}(sS)^\circ$ given, for example, in \cite[Theorem 3.5.3]{Car}. 
\end{proof}

\begin{lem}\label{lem-pairs-quotient}
Let $\Gamma$ be a connected complex reductive group, and let $S \subset \rZ_\Gamma^\circ$ be a torus contained in the center of $\Gamma$. Then the natural projection $\rho: \Gamma \to \Gamma/S$ induces a well-defined map $\cY(\Gamma)/_\sim \to \cY(\Gamma/S)/_\sim$. If $\Gamma/S$ is connected, semisimple, and adjoint, this map is injective.
\end{lem}

\begin{proof} It's not hard to see that if $(s, h) \in \cY(\Gamma)$, then $(sS, hS) \in \cY(\Gamma/S)$. Now suppose $(s, h) \sim_L (s, h')$ in $\cY(\Gamma)$, so there exists a maximal torus $T$ of $\rZ_\Gamma(s, h)$ and $\gamma \in Z_\Gamma(s)$ such that $\gamma h \gamma^{-1} h'^{-1} \in T$. We may choose a maximal torus $T'$ of $\rZ_{\Gamma/S}(sS, hS)$ such that $\rho(T) \subset T'$. Then $\rho(\gamma h\gamma^{-1} h'^{-1}) \in T'$, so $(sS, hS) \sim_L (sS, h'S)$. By symmetry, we see that $(s, h) \sim_R (s', h)$ in $\cY(\Gamma)$ implies $(sS, hS) \sim_R (s'S, hS)$ in $\cY(\Gamma/S)$. Thus $\rho$ induces a well-defined map $\cY(\Gamma)/\sim \to \cY(\Gamma/S)/\sim$. 

Suppose $(s, h) \in \cY(\Gamma)$. Then for any $z_1, z_2 \in S$, we have $(sz_1, hz_2) \in \cY(\Gamma)$. Since $hz_2 \in h\rZ_\Gamma(s)^\circ$, we have $(sz_1, hz_2) \sim_L (sz_1, h)$, and similarly $(sz_1, h) \sim_R (s, h)$. Thus $(sz_1, hz_1) \sim (s, h)$. 

Now suppose $\Gamma/S$ is connected, semisimple, and adjoint. Suppose $(sS, h'S) \in \cY(\Gamma/S)$ is in the image of $\cY(\Gamma)$ and that $(sS, hS) \sim_L (sS, h'S)$ in $\cY(\Gamma/S)$. Let $z_1, z_2, z_3, z_4 \in S$. To show injectivity, we must show that $(sz_1, hz_2) \sim (sz_3, h'z_4)$ in $\cY(\Gamma)$. We already know that $(sz_1, hz_2) \sim (s, h)$ and $(sz_3, h'z_4) \sim (s, h')$, so it suffices to show that $(s, h) \sim (s, h')$ in $\cY(\Gamma)$.

Note that $A_{\Gamma/S}(s)$ is abelian \cite[Proposition 2.1]{Re4}, and so by Lemma \ref{l:ell-comp}, $hh'^{-1} \in \rZ_{\Gamma/S}(sS)^\circ$. Then by Lemma \ref{lem-conn-centralizer} $hh^{-1} \in \rZ_\Gamma(s)^\circ$, so $(s, h) \sim_L (s, h')$. The same proof shows that if $(s', h) \in \cY(\Gamma)$ and $(sS, hS) \sim_R (s'S, hS)$ in $\cY(\Gamma/S)$, then $(s, h) \sim (s', h)$. Thus the map $\cY(\Gamma)/_\sim \to \cY(\Gamma/S)/_\sim$ is injective as claimed.
\end{proof}

Now fix a unipotent element $u \in G^\vee$ corresponding to the partition 
\begin{equation*}
\lambda=(\underbrace{a_1,\dots,a_1}_{r_1},\underbrace{a_2,\dots,a_2}_{r_2},\dots,\underbrace{a_\ell,\dots,a_\ell}_{r_\ell}), 
\end{equation*}
(with $\sum_{i=1}^\ell a_i r_i=n$).
Recall that
\begin{equation}
\Gamma_u \simeq \left(\prod_{i=1}^\ell \GL_{r_i}(\mathbb C)\right)/Z,
\end{equation}
where we think of $\GL_{r_i}$ as embedded block diagonally in $\GL_{(r_1 + \dots + r_\ell)}$, and 
$Z$ is the center of $\GL_{(r_1 + \dots +r_\ell)}(\bC)$. 
For each $i \in \{1, \dots, \ell\}$, let $h_i = \text{ diag }\{\zeta_{r_i}^{r_i - 1}\dots \zeta_{r_i}, 1\} \in \GL_{r_i}(\bC)$ and let $w_i$ be the matrix in $\GL_{r_i}(\bC)$ corresponding to the $r_i$-cycle $(1~2\dots r_i)$. Then $w_ih_iw_i^{-1} = \zeta_{r_i}h_i$. Note that if $\lam^\perp$ is the partition of $n$ dual to $\lam$, then $\gcd(\lam^\perp) = \gcd(r_1, \dots, r_\ell)$. 
For each $i \in \{1, \dots, \ell\}$, let $k_i = \frac{r_i}{\gcd(\lam^\perp)}$.
Let $w = \prod w_i^{k_i} \in \Gamma_u$, and let $h_0 = \prod h_i \in \Gamma_u$. Then $h_0$ commutes with $w$ since $w_i^kh_iw_i^{-k_i} = \zeta_{\gcd(\lam^\perp)}h_i$ for all $i$.

\begin{prop}
A complete set of compact pairs in $\Gamma_u$ is given by \[\{(h_0, 1), (h_0, w), \dots,\\ (h_0, w^{\gcd(\lam^\perp) - 1}) \}.\] In particular, $\Gamma_u$ contains $\gcd(\lam^\perp)$ compact pairs.
\end{prop}

\begin{proof}Let $\rho: \Gamma_u \to \prod_{i = 1}^\ell \PGL_{r_i}(\bC)$ be the natural projection, and let $S$ be the kernel of this projection. Using Lemma \ref{lem-pairs-pgln}, we can see that the compact pairs for $\prod_{i = 1}^\ell \PGL_{r_i}(\bC)$ are given by $\{(h_0, \prod_{i = 1}^\ell w_i^{c_i}) \mid c_i \in \{0, \dots, r_i - 1\} \}$. By Lemma \ref{lem-pairs-quotient}, the map $\cY(\Gamma_u)/_\sim \to \cY(\prod \PGL_{r_i}(\bC))/_\sim$ induced by $\rho$ is injective, and thus the pairs $(h_0, 1), (h_0, w), \dots, (h_0, w^{\gcd(\lam^\perp) - 1}) $ all determined distinct $\sim$-equivalence classes in $\cY(\Gamma_u)$. 

Next we claim that if $(s, h) \in \cY(\Gamma_u)$, then given $(xS, yS) \in \cY(\prod \GL_{r_i})$ with $(sS, hS) \sim (xS, yS)$, we have $xy = yx$, i.e. $(x, y) \in \cY(\Gamma_u)$. To see this, suppose $(sS, hS) \sim_L (sS, h'S)$. Then $h^{-1}h'S \in \rZ_{\prod \GL_{r_i}}(sS)^\circ$, so by Lemma \ref{lem-conn-centralizer}, $h^{-1}h' \in \rZ_{\Gamma_u}(s)^\circ$, which in particular implies that $h' \in \rZ_{\Gamma_u}(s)$ and $(s, h') \in \cY(\Gamma_u)$. Similarly, if $(sS, hS) \sim_R (s'S, hS)$, then $(s', h) \in \cY(\Gamma_u)$. Together these imply the claim.

Now suppose $(s, h) \in \cY(\Gamma_u)$ maps to the $\sim$-equivalence class determined by $(h_0, \prod w_i^{c_i}) \in \cY(\prod \PGL_{r_i}(\bC))$. Then by the claim in the previous paragraph, we see that $h_0$ and $\prod w_i^{c_i}$ commute, which implies $\zeta_{r_i}^{c_i} = \zeta_{r_j}^{c_j}$ for all $i, j \in \{1, \dots, \ell\}$, i.e. $\frac{c_i}{r_i} = \frac{c_j}{r_j}$ for all $i, j$ (here we're using the fact that $c_i < r_i$ for all $i$). Note that $k_i \mid c_i$ for all $i$. Indeed, fix $i$. Then $c_ir_j = c_jr_i$ for all $j$, so $c_ik_j = c_jk_i$ for all $j$, so $k_i \mid c_ik_j$ for all $j$. If $p$ is a prime with $p^m \mid k_i$, then by definition of $\gcd$, there exists $j$ such that $p \nmid k_j$, so $p^\ell \mid c_i$. Thus $k_i \mid c_i$. Writing $c_i = k_in_i$ for all $i$, we have $\frac{n_i}{\gcd(\lam^\perp)} = \frac{n_j}{\gcd(\lam^\perp)}$ for all $i, j$, and so $n_i = n_j$ for all $i, j$, and $\prod w_i^{c_i}$ is a power of $w$. Thus the compact pairs for $\Gamma_u$ are as claimed.
\end{proof}

We now consider Conjecture \ref{c:compact}. Fix a unipotent element $u \in G^\vee$ corresponding to $\lambda$ as above. To simplify notation, set $d=\gcd(\lam^\perp)$. The component group $A_{\Gamma_u}(h_0)=C_d$, generated by $w$ (mod $Z_{\Gamma_u}(h_0)^\circ$). The L-packet corresponding to the pair $(u,h_0)$ is $\{\pi(u,h_0,\phi)\mid \phi\in \widehat C_d\}.$ As it is well known, see for example \cite[Theorem 4.3]{GK}, the representations in this L-packet are precisely the direct summands (appearing with multiplicity one) in the restriction from $\GL_n(F)$ to $\SL_n(F)$ of the irreducible tempered $\GL_n(F)$-representation 
\[
\pi(u,h_0)=i_{\prod_{i=1}^\ell \GL_{a_i}(F)^{r_i}}^{\GL_n(F)}(\St\otimes |\det|^{h_0}).
\]
Moreover, the component group $C_d$ acts simply transitively on the $\SL_n(F)$ L-packet. Set 
\begin{equation}
    \Pi(u,h_0,w^m)=\sum_{\phi\in \widehat C_d}\phi(w^m)~ \pi(u,h_0,\phi). 
\end{equation}
We are interested in the parahoric restriction of $\Pi(u,h_0,w^m)$. There are $n$ conjugacy classes of maximal parahoric subgroups in $\SL_n(F)$, all conjugate in $\GL_n(F)$ to $K_0=\SL_n(\mathfrak o_F).$ This is an easy generalization of the calculation in \cite[\S13.3]{ACR} in the case $\ell=1$ (the elliptic case). As in {\it loc. cit.}, denote by $W^a$ the affine Weyl group of type $A_{n-1}$ and let $W_i\cong S_n$, $i=0,\dots,n-1$, denote the finite parahoric subgroups of $W^a$ corresponding to the maximal parahoric subgroups of $G$, such that $W_0=W$ corresponds to $K_0$.

Fix a primitive $d$-th root $\zeta_d$ of $1$, and denote by $\phi_k\in \widehat C_d$ the character $\phi_k(w)=\zeta_d^k$, $k=0,\dots,d-1$. Then
\[
    \Pi(u,h_0,w^m)=\sum_{k=0}^{d-1}\zeta_d^{mk}~ \pi(u,h_0,\phi_k),\quad 0\le m\le d-1.
\]
Since the order pairs $(w^m,h_0)$ and $h_0,w^{-m})$ are $\Gamma_u$-conjugate, we have
\[
\FT_\cpt^\vee(\Pi(u,h_0,w^m))=\Pi(u,h_0,w^{-m}).
\]
Denote 
\[
W_u=\left(\prod_{i=1^\ell} S_{a_i}^{r_i}\right)\rtimes C_d\le S_n.
\]
As in \cite[\S13.3]{ACR}, for computing the parahoric restrictions, we may think of the Iwahori-fixed vectors of the representation $\pi(u,h_0,\phi)$ (a module for the Iwahori-Hecke algebra) and its deformation $q\to 1$, $\sigma_0(\pi(u,h_0,\phi)^I)$, a $W^a$-representation. In this equivalent setting, we need to determine the restriction of 
\[\sigma_0(\pi(u,h_0,\phi)^I)=\Ind_{W_u\rtimes X}^{S_n\rtimes X}((\sgn\rtimes \phi)\otimes h_0)\]
to the finite parahoric Weyl groups $W_i$. Here $X$ is the root lattice of $G$, so that $W_a=W_0\ltimes X=S_n\ltimes X$, and $h_0$ is regarded as a character of $X$. A similar calculation with Mackey restriction as in \cite[Lemma 13.8]{ACR} gives
\begin{equation}
    \Ind_{W_u\rtimes X}^{S_n\rtimes X}((\sgn\rtimes \phi_k)\otimes h_0)|_{W_i}\cong \Ind_{W_u}^{S_n}(\sgn \rtimes \phi_{k+\lfloor d i/n \rfloor}), \quad 0\le i<n,\ 0\le k<d.
\end{equation}
As in the proof of \cite[Proposition 13.9]{ACR}
\begin{align*}
\res_{K_i}\sigma_0(\Pi(u,h_0,w^m)^I)&=\zeta_d^{-mj}\sum_{k\in \bZ/d} \zeta_d^{mk} \Ind_{W_u}^{S_n}(\sgn\rtimes \phi_k),\text { where } j=\lfloor {d i}/n\rfloor, \text{ while}\\
\res_{K_i}\sigma_0(\Pi(u,h_0,w^{-m})^I)&=\zeta_d^{mj}\sum_{k\in \bZ/d} \zeta_d^{mk} \Ind_{W_u}^{S_n}(\sgn\rtimes \phi_k),
\end{align*}
implying that, analogously with {\it loc. cit.},
\begin{equation}
    \res_{K_i}\circ \FT^\vee_\cpt (\Pi(u,h_0,w^m))=\zeta_d^{2m \lfloor {d i}/n\rfloor} \FT_\cpt\circ \res_{K_i} (\Pi(u,h_0,w^m)),\quad 0\le m<d.
\end{equation}

We have thus proved the following:

\begin{thm}\label{t:SLn}
Conjecture \ref{c:compact} holds when $\bG = \SL_n$.
\end{thm}

\subsection{$\PGL_n(F)$}
Now we take $\bG = \PGL_n$. 
First we we classify compact pairs for $G^\vee = \SL_n(\bC)$. Suppose $u \in G^\vee$ is unipotent corresponding to the partition $\lambda=(\underbrace{a_1,\dots,a_1}_{r_1},\underbrace{a_2,\dots,a_2}_{r_2},\dots,\underbrace{a_\ell,\dots,a_\ell}_{r_\ell})$. Recall that we have
\begin{equation}\label{eqn-GammauPGL}
\Gamma_u \simeq \left\{ (x_1, \dots, x_\ell) \in \prod_{i = 1}^\ell \GL_{r_i}( \bC) \mid \prod_{i = 1}^\ell \det(x_i)^{a_i} = 1\right\}.
\end{equation}
Let $\zeta \in \bC$ be a primitive $n$th root of unity, and let $z = \zeta I_n \in \PGL_n$ be the corresponding scalar matrix. 

\begin{prop}
With notation as above, let $d = \gcd(\lam)$. Then a set of representatives for $\cY(\Gamma_u)/_\sim$ is given by $\{(z^{m_1}, z^{m_2}) \mid m_1, m_2 \in \{0, 1, \dots, d - 1\}\}$.
In particular, the number of compact pairs in $\Gamma_u$ is $d^2$. 
\end{prop}

\begin{proof}
The component group of $\Gamma_u$ is $\bZ/d\bZ$. 
There is a surjective map $f_u: \Gamma_u \to \mu_d$ defined by
\begin{equation*}
(x_1, \dots, x_\ell) \mapsto \prod_{i = 1}^\ell \det(x_i)^{\frac{a_i}{d}}.
\end{equation*}
The kernel of $f_u$ is the identity component of $\Gamma_u$.
Note that 
$z = (\zeta I_{r_1}, \dots, \zeta I_{r_\ell}) \in \Gamma_u$, and that $f_u(z)$ generates $\mu_d$. 
The pairs $(z^{m_1}, z^{m_2})$ are all in distinct $\sim$-equivalence classes for distinct pairs $(m_1, m_2)$ in $\{0, \dots, d - 1\}$: this follows easily from the definition of $\sim$.  
Now suppose $(s, h) \in \mathcal{Y}(\Gamma_u)$. After conjugating by an element of $\SL_{r_1}(\bC) \times \dots \times \SL_{r_\ell}(\bC) \subset \Gamma_u$, we may assume $s$ and $h$ are diagonal. If $h \in z^{m_2}\Gamma_u^\circ$, then $(s, h) \sim_L (s, z^{m_2})$ by definition of $\sim_L$, and similarly if $s \in z^{m_1}\Gamma_u^\circ$, then $(s, z^{m_2}) \sim_R (z^{m_1}, z^{m_2})$. \end{proof}

Fix $u$ as above. Now we consider the sums $\Pi(u, s, h)$. We take 
$M^\vee \simeq S((\GL_{a_1}(\bC))^{r_1} \times \dots \times (\GL_{a_\ell}(\bC))^{r_\ell})$ 
to be the Levi subgroup of $G^\vee$ such that $u$ is regular in $M^\vee$. 
Let $A_u$ be the component group of $\Gamma_u$, which is isomorphic to $\bZ/d\bZ$. 
We have a surjective homomorphism $Z_{G^\vee} \to A_u$, so every character $\rho$ of $A_u$ determines a character of $Z_{G^\vee}$ and thus determines a unique pure inner twist $G_\rho$ of $G$. Further, since $Z_{M^\vee}/Z_{M^\vee}^\circ \simeq A_u$, each of these characters determines a pure inner twist $M_\rho$ of $M$, such that $M_\rho$ is a Levi subgroup of $M$.
Then 
\begin{align*}
\Pi(u, z^{m_1}, z^{m_2}) &= \sum_{\rho \in \hat{A}_u} \rho(z^{m_2})\pi(u, z^{m_1}, \rho)\\
&= \sum_{\rho \in \hat{A}_u} \rho(z^{m_2})\pi(u, 1, \rho) \otimes \chi_{m_1, \rho}\\
&= \sum_{\rho \in \hat{A}_u} \rho(z^{m_2})i_{M_\rho}^{G_\rho}(\St_{M_\rho}) \otimes \chi_{m_1, \rho}
\end{align*}
where $\St_{M_\rho}$ is the Steinberg representation of $M_\rho$, and $\chi_{m_1, \rho}$ is the weakly unramified character of $G_\rho$ corresponding to $z^{m_1}$.

\begin{lem}\label{l:FT-PGL}
Let $H = (\GL_r(\bF_q))^m/\bF_q^\times \rtimes \bZ/m\bZ$ for some positive integers $r, m$, where $\bZ/m\bZ$ acts by cyclically permuting the factors $\GL_r(\bF_q)$. Let $H^\circ = (\GL_r(\bF_q))^m/\bF_q^\times$. For every irreducible unipotent representation $\rho$ of $H^\circ$, the representation $\Ind_{H^\circ}^H \rho$ is fixed under $\FT_H$.
\end{lem}

\begin{proof}
We may write $\rho$ as $\rho_1 \boxtimes \dots \boxtimes \rho_m$ for some irreducible unipotent representations $\rho_1, \dots \rho_m$ of $\GL_r(\bF_q)$. Then $\bZ/m\bZ$ permutes the factors $\rho_i$. Let $\bZ/c\bZ \subset \bZ/m\bZ$ be the stabilizer of $\rho$ under this action. Given a character $\chi$ of $\bZ/c\bZ$, we may extend $\rho$ to a representation $V_\chi$ of $H^\circ \rtimes \bZ/c\bZ$, such that $\bZ/c\bZ$ acts via $\chi$. Then $\Ind_{H^\circ}^H \rho$ decomposes into a sum of irreducible subrepresentations as
\begin{equation*}
\Ind_{H^\circ}^H \rho = \bigoplus_{\chi \in \Irr(\bZ/c\bZ)} \Ind_{H^\circ \rtimes \bZ/c\bZ}^H V_\chi.
\end{equation*}
In the parametrization of irreducible unipotent representations of $H$, $\rho$ corresponds to a family with group $\bZ/c\bZ$, and each representation $\Ind_{H^\circ \rtimes \bZ/c\bZ}^H V_\chi$ corresponds to $(0, \chi) \in \mathcal{M}(\bZ/c\bZ)$ (cf. \cite[Example 6.4]{ACR}, which also explains notation). The claim can then be checked by a straightforward computation using the definition of $\FT_H$.
\end{proof}

\begin{prop}\label{e:hyperspecial-PGL}
Let $G = \PGL_n$. We have 
\begin{equation*}
\res_K \FT^\vee_{\cpt} \Pi(u, s, h) = \FT_K \res_K \Pi(u, s, h)
\end{equation*}
in each of the following cases:
\begin{itemize}
\item[(1)] when $K$ is a hyperspecial maximal parahoric subgroup
\item[(2)] when $u$ corresponds to a partition $\lam$ such that $\gcd(\lam) = 1$.
\end{itemize}
\end{prop}

\begin{proof}
First assume $K$ is a hyperspecial parahoric subgroup. Then $\overline{K}$ is connected (and so has no inner twists), so $\res_K i_{M_\rho}^{G_\rho}(\St_{M_\rho})$ is zero unless $\rho = 1$. We have
\begin{equation*}
\res_K \Pi(u, z^{m_1}, z^{m_2}) = \res_K i_M^G (\St_M) \otimes \chi_{m_1}.
\end{equation*}
Since $\chi_{m_1}$ is trivial on $K$, this is equal to $\res_K i_M^G (\St_M)$. In particular, it only depends on $u$ and not on the compact pair $(z^{m_1}, z^{m_2})$. Since $\FT_{\overline{K}}$ is trivial, we see that Diagram \ref{e:diagram-K-cpt} commutes for this choice of $K$, i.e. that $\res_K \FT^\vee_{\cpt} \Pi(u, s, h) = \FT_K \res_K \Pi(u, s, h)$ as claimed.

Now suppose $\gcd(\lam) = 1$. This tells us that $\rZ_{M^\vee}$ is connected. Thus $M$ has no non-split pure inner twists, and every maximal compact subgroup of $M$ is a parahoric. The only compact pair in $\Gamma_u$ is $(1, 1)$. For any maximal compact subgroup $K$ of $G$, we have 
\begin{equation*}
\res_K \Pi(u, 1, 1) = \sum_{g \in K \setminus G/P} i_{\overline{K \cap P^g}}^{\overline{K}} \res_{K \cap M^g}^{M^g} (\St_{M})^g.
\end{equation*}
We claim that for each $g$, the representation $i_{\overline{K \cap P^g}}^{\overline{K}} \res_{K \cap M^g}^{M^g} (\St_{M})^g$ is fixed under $\FT_K$. To see this, note that $\overline{K} = (\GL_r(\bF_q))^m/\bF_q^\times \rtimes \bZ/m\bZ$ for some $m$ and $r$. 
Also note that $\overline{K \cap P^g} \subset \overline{K}^\circ$, so we may view $i_{\overline{K \cap P^g}}^{\overline{K}} \res_{K \cap M^g}^{M^g} (\St_{M})^g$ as $\Ind_{\overline{K}^\circ}^{\overline{K}} (i_{\overline{K \cap P^g}}^{\overline{K}^\circ} \res_{K \cap M^g}^{M^g} (\St_{M})^g))$. The claim then follows from Lemma \ref{l:FT-PGL}.
\end{proof}

\begin{thm}\label{t:PGL}
Conjecture \ref{c:compact} holds for $\PGL_n$ whenever $n$ is prime. More specifically, Diagram \ref{e:diagram-cpt} commutes with trivial roots of unity.
\end{thm}

\begin{proof} 
If $u$ is regular, this follows from Example \ref{ex:reg-cpt}, and if $u$ is not regular, then this follows from Proposition \ref{e:hyperspecial-PGL}. 
\end{proof}

\begin{rem}
Conjecture \ref{c:compact} becomes significantly more complicated for $\PGL_n$ when $n$ is not prime. In this case, 
$\res_K \Pi(u, s, h)$ will generally be a sum of parabolically induced reduced representations for inner twists of $\overline{K}$, and $\FT_K$ mixes the spaces $R_\un(\overline{K}')$ corresponding to the different inner twists $\overline{K}'$. It then becomes necessary to understand and relate the double-coset spaces $K_\rho\setminus G_\rho/P_\rho$ as $\rho$ ranges over $\widehat{A}_u$ (here $P_\rho$ is a parabolic with Levi component $M_\rho$ and $K_\rho \subset G_\rho$ is a maximal compact subgroup such that $\overline{K}_\rho$ is an inner twist of $\overline{K}$). 
Similar considerations arise for other non-simply connected groups $G$. One can approach these issues using ideas in \cite[Section 4.5]{Hai09}.
\end{rem}

\bibliographystyle{plain}
\bibliography{compact-FT}

@article {ACR,
    AUTHOR = {Aubert, Anne-Marie and Ciubotaru, Dan and Romano, Beth},
     TITLE = {A nonabelian {F}ourier transform for tempered unipotent
              representations},
   JOURNAL = {Compos. Math.},
  FJOURNAL = {Compositio Mathematica},
    VOLUME = {161},
      YEAR = {2025},
    NUMBER = {1},
     PAGES = {13--73},
      ISSN = {0010-437X,1570-5846},
   MRCLASS = {22E50 (20C33)},
  MRNUMBER = {4877237},
MRREVIEWER = {Alain\ Valette},
       DOI = {10.1112/S0010437X24007401},
       URL = {https://doi.org/10.1112/S0010437X24007401},
}

@article {AMS1,
    AUTHOR = {Aubert, Anne-Marie and Moussaoui, Ahmed and Solleveld,
              Maarten},
     TITLE = {Generalizations of the {S}pringer correspondence and cuspidal
              {L}anglands parameters},
   JOURNAL = {Manuscripta Math.},
  FJOURNAL = {Manuscripta Mathematica},
    VOLUME = {157},
      YEAR = {2018},
    NUMBER = {1-2},
     PAGES = {121--192},
      ISSN = {0025-2611,1432-1785},
   MRCLASS = {11S37 (20G05 22E50)},
  MRNUMBER = {3845761},
MRREVIEWER = {Volker\ J.\ Heiermann},
       DOI = {10.1007/s00229-018-1001-8},
       URL = {https://doi.org/10.1007/s00229-018-1001-8},
}

@article {BNP,
    AUTHOR = {Ben-Zvi, David and Nadler, David and Preygel, Anatoly},
     TITLE = {A spectral incarnation of affine character sheaves},
   JOURNAL = {Compositio Math.},
    VOLUME = {153},
      YEAR = {2017},
    NUMBER = {9},
     PAGES = {1908--1944}}

@article {LNY,
    AUTHOR = {Li, Penghui and Nadler, David and Yun, Zhiwei},
     TITLE = {Functions on the commuting stack via {L}anglands duality},
   JOURNAL = {Annals Math.},
    VOLUME = {200},
      YEAR = {2024},
    NUMBER = {2},
     PAGES = {609--748}}

@article {BDK,
    AUTHOR = {Bernstein, Joseph and Deligne, Pierre and Kazhdan, David},
     TITLE = {Trace {P}aley-{W}iener theorem},
   JOURNAL = {J. Analyse Math.},
    VOLUME = {47},
      YEAR = {1986},
     PAGES = {180--192}}

@article {BKK,
    AUTHOR = {Bezrukavnikov, Roman and Karpov, Ivan and Krylov, Vasily},
     TITLE = {A geometric realization of the asymptotic affine {H}ecke algebra},
     YEAR = {2024},
   JOURNAL = {arXiv:2312.10582v2},
}

@inproceedings {Bor,
    AUTHOR = {Borel, A.},
     TITLE = {Automorphic {$L$}-functions},
 BOOKTITLE = {Automorphic forms, representations and {$L$}-functions
              ({P}roc. {S}ympos. {P}ure {M}ath., {O}regon {S}tate {U}niv.,
              {C}orvallis, {O}re., 1977), {P}art 2},
    SERIES = {Proc. Sympos. Pure Math., XXXIII},
     PAGES = {27--61},
 PUBLISHER = {Amer. Math. Soc., Providence, RI},
      YEAR = {1979},
      ISBN = {0-8218-1437-0},
   MRCLASS = {10D40 (12A67 22E50)},
  MRNUMBER = {546608},
MRREVIEWER = {Yasuhiro\ Asoo},
}

@article {BK,
    AUTHOR = {Bushnell, Colin J. and Kutzko, Philip C.},
     TITLE = {Smooth representations of reductive {$p$}-adic groups:
              structure theory via types},
   JOURNAL = {Proc. London Math. Soc. (3)},
  FJOURNAL = {Proceedings of the London Mathematical Society. Third Series},
    VOLUME = {77},
      YEAR = {1998},
    NUMBER = {3},
     PAGES = {582--634},
      ISSN = {0024-6115,1460-244X},
   MRCLASS = {22E50 (22E35)},
  MRNUMBER = {1643417},
MRREVIEWER = {David\ Goldberg},
       DOI = {10.1112/S0024611598000574},
       URL = {https://doi.org/10.1112/S0024611598000574},
}

@article {BT1,
    AUTHOR = {Borel, Armand and Tits, Jacques},
     TITLE = {Groupes r\'{e}ductifs},
   JOURNAL = {Inst. Hautes \'{E}tudes Sci. Publ. Math.},
  FJOURNAL = {Institut des Hautes \'{E}tudes Scientifiques. Publications
              Math\'{e}matiques},
    NUMBER = {27},
      YEAR = {1965},
     PAGES = {55--150},
      ISSN = {0073-8301,1618-1913},
   MRCLASS = {14.50 (20.75)},
  MRNUMBER = {207712},
MRREVIEWER = {F.\ D.\ Veldkamp},
       URL = {http://www.numdam.org/item?id=PMIHES_1965__27__55_0},
}

@book {Car,
    AUTHOR = {Carter, Roger W.},
     TITLE = {Finite groups of {L}ie type},
    SERIES = {Pure and Applied Mathematics (New York)},
      NOTE = {Conjugacy classes and complex characters,
              A Wiley-Interscience Publication},
 PUBLISHER = {John Wiley \& Sons, Inc., New York},
      YEAR = {1985},
     PAGES = {xii+544},
      ISBN = {0-471-90554-2},
   MRCLASS = {20G40 (20-02 20C15)},
  MRNUMBER = {794307},
MRREVIEWER = {David\ B.\ Surowski},
}

@book {Cas,
    AUTHOR = {Casselman, W.},
     TITLE = {Introduction to the theory of admissible representations of $p$-adic reductive groups},
    SERIES = {https://www.math.ubc.ca/$\sim$cass/research/pdf/p-adic-book.pdf},
      NOTE = {},
 PUBLISHER = {},
      YEAR = {1995},
    URL = {https://www.math.ubc.ca/~cass/research/pdf/p-adic-book.pdf},
}

@article {Ciu,
    AUTHOR = {Ciubotaru, Dan},
     TITLE = {The nonabelian {F}ourier transform for elliptic unipotent representations of exceptional $p$-adic groups},
     YEAR = {2020},
   JOURNAL = {arXiv:2006.13540v2},
}

@article {Ci-inv,
    AUTHOR = {Ciubotaru, Dan},
     TITLE = {Types and unitary representations of reductive {$p$}-adic
              groups},
   JOURNAL = {Invent. Math.},
  FJOURNAL = {Inventiones Mathematicae},
    VOLUME = {213},
      YEAR = {2018},
    NUMBER = {1},
     PAGES = {237--269},
      ISSN = {0020-9910,1432-1297},
   MRCLASS = {22E50},
  MRNUMBER = {3815566},
MRREVIEWER = {Jeffrey\ D.\ Adler},
       DOI = {10.1007/s00222-018-0790-4},
       URL = {https://doi.org/10.1007/s00222-018-0790-4},
}

@article {CH2,
    AUTHOR = {Ciubotaru, Dan and He, Xuhua},
     TITLE = {Cocenters of {$p$}-adic groups, {III}: {E}lliptic and rigid
              cocenters},
   JOURNAL = {Peking Math. J.},
  FJOURNAL = {Peking Mathematical Journal},
    VOLUME = {4},
      YEAR = {2021},
    NUMBER = {2},
     PAGES = {159--186},
      ISSN = {2096-6075,2524-7182},
   MRCLASS = {22E50 (20C08)},
  MRNUMBER = {4319924},
MRREVIEWER = {Marko\ Tadi\'{c}},
       DOI = {10.1007/s42543-020-00027-1},
       URL = {https://doi.org/10.1007/s42543-020-00027-1},
}

@book {CM,
    AUTHOR = {Collingwood, David H. and McGovern, William M.},
     TITLE = {Nilpotent orbits in semisimple {L}ie algebras},
    SERIES = {Van Nostrand Reinhold Mathematics Series},
 PUBLISHER = {Van Nostrand Reinhold Co., New York},
      YEAR = {1993},
     PAGES = {xiv+186},
      ISBN = {0-534-18834-6},
   MRCLASS = {17-02 (17B20 17B25 22E60)},
  MRNUMBER = {1251060},
MRREVIEWER = {Stephen\ Slebarski},
}

@article {Da,
    AUTHOR = {Dat, J.-F.},
     TITLE = {On the {$K_0$} of a {$p$}-adic group},
   JOURNAL = {Invent. Math.},
  FJOURNAL = {Inventiones Mathematicae},
    VOLUME = {140},
      YEAR = {2000},
    NUMBER = {1},
     PAGES = {171--226},
      ISSN = {0020-9910,1432-1297},
   MRCLASS = {22E50 (19A99)},
  MRNUMBER = {1779801},
MRREVIEWER = {Bertrand\ Lemaire},
       DOI = {10.1007/s002220050360},
       URL = {https://doi.org/10.1007/s002220050360},
}

@article {DM,
    AUTHOR = {Digne, Fran\c{c}ois and Michel, Jean},
     TITLE = {On {L}usztig's parametrization of characters of finite groups
              of {L}ie type},
   JOURNAL = {Ast\'{e}risque},
  FJOURNAL = {Ast\'{e}risque},
    NUMBER = {181-182},
      YEAR = {1990},
     PAGES = {6, 113--156},
      ISSN = {0303-1179,2492-5926},
   MRCLASS = {20G05 (20C15)},
  MRNUMBER = {1051245},
MRREVIEWER = {Bhama\ Srinivasan},
}

@article {DM2,
    AUTHOR = {Digne, Fran\c{c}ois and Michel, Jean},
     TITLE = {Quasi-semisimple elements},
   JOURNAL = {Proc. Lond. Math. Soc. (3)},
  FJOURNAL = {Proceedings of the London Mathematical Society. Third Series},
    VOLUME = {116},
      YEAR = {2018},
    NUMBER = {5},
     PAGES = {1301--1328},
      ISSN = {0024-6115,1460-244X},
   MRCLASS = {20G15},
  MRNUMBER = {3805058},
MRREVIEWER = {Giancarlo\ Lucchini Arteche},
       DOI = {10.1112/plms.12121},
       URL = {https://doi.org/10.1112/plms.12121},
}

@book {DM91,
    AUTHOR = {Digne, Fran\c{c}ois and Michel, Jean},
     TITLE = {Representations of finite groups of {L}ie type},
    SERIES = {London Mathematical Society Student Texts},
    VOLUME = {21},
 PUBLISHER = {Cambridge University Press, Cambridge},
      YEAR = {1991},
     PAGES = {iv+159},
      ISBN = {0-521-40117-8; 0-521-40648-X},
   MRCLASS = {20G05 (20-02 20C33 20G40)},
  MRNUMBER = {1118841},
MRREVIEWER = {Bhama\ Srinivasan},
       DOI = {10.1017/CBO9781139172417},
       URL = {https://doi.org/10.1017/CBO9781139172417},
}

@article {GHKR,
    AUTHOR = {G\"{o}rtz, Ulrich and Haines, Thomas J. and Kottwitz, Robert
              E. and Reuman, Daniel C.},
     TITLE = {Affine {D}eligne-{L}usztig varieties in affine flag varieties},
   JOURNAL = {Compos. Math.},
  FJOURNAL = {Compositio Mathematica},
    VOLUME = {146},
      YEAR = {2010},
    NUMBER = {5},
     PAGES = {1339--1382},
      ISSN = {0010-437X,1570-5846},
   MRCLASS = {14L35 (11S25 14F30 20G25)},
  MRNUMBER = {2684303},
MRREVIEWER = {Alan\ Koch},
       DOI = {10.1112/S0010437X10004823},
       URL = {https://doi.org/10.1112/S0010437X10004823},
}

@article {GK,
    AUTHOR = {Gelbart, S. S. and Knapp, A. W.},
     TITLE = {{$L$}-indistinguishability and {$R$} groups for the special
              linear group},
   JOURNAL = {Adv. in Math.},
  FJOURNAL = {Advances in Mathematics},
    VOLUME = {43},
      YEAR = {1982},
    NUMBER = {2},
     PAGES = {101--121},
      ISSN = {0001-8708},
   MRCLASS = {22E50 (12B27)},
  MRNUMBER = {644669},
MRREVIEWER = {S.\ I.\ Gel\cprime fand},
       DOI = {10.1016/0001-8708(82)90030-5},
       URL = {https://doi.org/10.1016/0001-8708(82)90030-5},
}

@book {GM,
    AUTHOR = {Geck, Meinolf and Malle, Gunter},
     TITLE = {The character theory of finite groups of {L}ie type},
    SERIES = {Cambridge Studies in Advanced Mathematics},
    VOLUME = {187},
      NOTE = {A guided tour},
 PUBLISHER = {Cambridge University Press, Cambridge},
      YEAR = {2020},
     PAGES = {ix+394},
      ISBN = {978-1-108-48962-1},
   MRCLASS = {20C33 (20D06 20G05)},
  MRNUMBER = {4211779},
MRREVIEWER = {Donald\ L.\ White},
}

@article {Hai09,
    AUTHOR = {Haines, Thomas J.},
     TITLE = {The base change fundamental lemma for central elements in
              parahoric {H}ecke algebras},
   JOURNAL = {Duke Math. J.},
  FJOURNAL = {Duke Mathematical Journal},
    VOLUME = {149},
      YEAR = {2009},
    NUMBER = {3},
     PAGES = {569--643},
      ISSN = {0012-7094,1547-7398},
   MRCLASS = {22E50 (20G25)},
  MRNUMBER = {2553880},
MRREVIEWER = {Anne-Marie\ H.\ Aubert},
       DOI = {10.1215/00127094-2009-045},
       URL = {https://doi.org/10.1215/00127094-2009-045},
}

@incollection {Hai14,
    AUTHOR = {Haines, Thomas J.},
     TITLE = {The stable {B}ernstein center and test functions for {S}himura
              varieties},
 BOOKTITLE = {Automorphic forms and {G}alois representations. {V}ol. 2},
    SERIES = {London Math. Soc. Lecture Note Ser.},
    VOLUME = {415},
     PAGES = {118--186},
 PUBLISHER = {Cambridge Univ. Press, Cambridge},
      YEAR = {2014},
      ISBN = {978-1-107-69363-0},
   MRCLASS = {11G18 (20G25 22E35 22E50)},
  MRNUMBER = {3444233},
MRREVIEWER = {Shuichiro\ Takeda},
}

@article {HS,
    AUTHOR = {Hiraga, Kaoru and Saito, Hiroshi},
     TITLE = {On {$L$}-packets for inner forms of {$SL_n$}},
   JOURNAL = {Mem. Amer. Math. Soc.},
  FJOURNAL = {Memoirs of the American Mathematical Society},
    VOLUME = {215},
      YEAR = {2012},
    NUMBER = {1013},
     PAGES = {vi+97},
      ISSN = {0065-9266,1947-6221},
      ISBN = {978-0-8218-5364-1},
   MRCLASS = {22E55 (11F70)},
  MRNUMBER = {2918491},
MRREVIEWER = {Wee\ Teck\ Gan},
       DOI = {10.1090/S0065-9266-2011-00642-8},
       URL = {https://doi.org/10.1090/S0065-9266-2011-00642-8},
}

@article {IM,
    AUTHOR = {Iwahori, N. and Matsumoto, H.},
     TITLE = {On some {B}ruhat decomposition and the structure of the
              {H}ecke rings of {${\germ p}$}-adic {C}hevalley groups},
   JOURNAL = {Inst. Hautes \'{E}tudes Sci. Publ. Math.},
  FJOURNAL = {Institut des Hautes \'{E}tudes Scientifiques. Publications
              Math\'{e}matiques},
    NUMBER = {25},
      YEAR = {1965},
     PAGES = {5--48},
      ISSN = {0073-8301,1618-1913},
   MRCLASS = {20.70 (14.50)},
  MRNUMBER = {185016},
MRREVIEWER = {Rimhak\ Ree},
       URL = {http://www.numdam.org/item?id=PMIHES_1965__25__5_0},
}

@article {Kaz,
    AUTHOR = {Kazhdan, David},
     TITLE = {Cuspidal geometry of {$p$}-adic groups},
   JOURNAL = {J. Analyse Math.},
  FJOURNAL = {Journal d'Analyse Math\'{e}matique},
    VOLUME = {47},
      YEAR = {1986},
     PAGES = {1--36},
      ISSN = {0021-7670,1565-8538},
   MRCLASS = {22E50},
  MRNUMBER = {874042},
MRREVIEWER = {Ernst-Wilhelm\ Zink},
       DOI = {10.1007/BF02792530},
       URL = {https://doi.org/10.1007/BF02792530},
}

@article {Ku,
    AUTHOR = {Kutzko, P. C.},
     TITLE = {Mackey's theorem for nonunitary representations},
   JOURNAL = {Proc. Amer. Math. Soc.},
  FJOURNAL = {Proceedings of the American Mathematical Society},
    VOLUME = {64},
      YEAR = {1977},
    NUMBER = {1},
     PAGES = {173--175},
      ISSN = {0002-9939,1088-6826},
   MRCLASS = {22D30 (22E50)},
  MRNUMBER = {442145},
MRREVIEWER = {A.\ Kleppner},
       DOI = {10.2307/2041005},
       URL = {https://doi.org/10.2307/2041005},
}

@book {Lubook,
    AUTHOR = {Lusztig, George},
     TITLE = {Characters of reductive groups over a finite field},
    SERIES = {Annals of Mathematics Studies},
    VOLUME = {107},
 PUBLISHER = {Princeton University Press, Princeton, NJ},
      YEAR = {1984},
     PAGES = {xxi+384},
      ISBN = {0-691-08350-9; 0-691-08351-7},
   MRCLASS = {20G05 (14L20 20C15)},
  MRNUMBER = {742472},
MRREVIEWER = {Bhama\ Srinivasan},
       DOI = {10.1515/9781400881772},
       URL = {https://doi.org/10.1515/9781400881772},
}

@article {Lu86,
    AUTHOR = {Lusztig, George},
     TITLE = {Character sheaves. {IV}},
   JOURNAL = {Adv. in Math.},
  FJOURNAL = {Advances in Mathematics},
    VOLUME = {59},
      YEAR = {1986},
    NUMBER = {1},
     PAGES = {1--63},
      ISSN = {0001-8708},
   MRCLASS = {20G05 (22E47)},
  MRNUMBER = {825086},
MRREVIEWER = {Bhama\ Srinivasan},
       DOI = {10.1016/0001-8708(86)90036-8},
       URL = {https://doi.org/10.1016/0001-8708(86)90036-8},
}

@article {LuI,
    AUTHOR = {Lusztig, George},
     TITLE = {Classification of unipotent representations of simple
              {$p$}-adic groups},
   JOURNAL = {Internat. Math. Res. Notices},
  FJOURNAL = {International Mathematics Research Notices},
      YEAR = {1995},
    NUMBER = {11},
     PAGES = {517--589},
      ISSN = {1073-7928,1687-0247},
   MRCLASS = {22E50 (20G25 22E35)},
  MRNUMBER = {1369407},
MRREVIEWER = {Lawrence\ Morris},
       DOI = {10.1155/S1073792895000353},
       URL = {https://doi.org/10.1155/S1073792895000353},
}

@article {MW,
    AUTHOR = {M{\oe}glin, Colette and Waldspurger, Jean-Loup},
     TITLE = {Paquets stables de repr\'{e}sentations temp\'{e}r\'{e}es et de
              r\'{e}duction unipotente pour {${\rm SO}(2n+1)$}},
   JOURNAL = {Invent. Math.},
  FJOURNAL = {Inventiones Mathematicae},
    VOLUME = {152},
      YEAR = {2003},
    NUMBER = {3},
     PAGES = {461--623},
      ISSN = {0020-9910,1432-1297},
   MRCLASS = {22E50},
  MRNUMBER = {1988295},
       DOI = {10.1007/s00222-002-0274-3},
       URL = {https://doi.org/10.1007/s00222-002-0274-3},
}

@article {Pra1,
    AUTHOR = {Prasad, Dipendra},
     TITLE = {Generalizing the {MVW} involution, and the contragredient},
   JOURNAL = {Trans. Amer. Math. Soc.},
  FJOURNAL = {Transactions of the American Mathematical Society},
    VOLUME = {372},
      YEAR = {2019},
    NUMBER = {1},
     PAGES = {615--633},
      ISSN = {0002-9947,1088-6850},
   MRCLASS = {11F70 (22E55)},
  MRNUMBER = {3968781},
MRREVIEWER = {Kimball\ L.\ Martin},
       DOI = {10.1090/tran/7602},
       URL = {https://doi.org/10.1090/tran/7602},
}

@article {Re2,
    AUTHOR = {Reeder, Mark},
     TITLE = {Formal degrees and {$L$}-packets of unipotent discrete series
              representations of exceptional {$p$}-adic groups},
      NOTE = {With an appendix by Frank L\"{u}beck},
   JOURNAL = {J. Reine Angew. Math.},
  FJOURNAL = {Journal f\"{u}r die Reine und Angewandte Mathematik. [Crelle's
              Journal]},
    VOLUME = {520},
      YEAR = {2000},
     PAGES = {37--93},
      ISSN = {0075-4102,1435-5345},
   MRCLASS = {22E50 (11F70 22E35)},
  MRNUMBER = {1748271},
MRREVIEWER = {Dihua\ Jiang},
       DOI = {10.1515/crll.2000.023},
       URL = {https://doi.org/10.1515/crll.2000.023},
}

@article {Re4,
    AUTHOR = {Reeder, Mark},
     TITLE = {Torsion automorphisms of simple {L}ie algebras},
   JOURNAL = {Enseign. Math. (2)},
  FJOURNAL = {L'Enseignement Math\'{e}matique. Revue Internationale. 2e
              S\'{e}rie},
    VOLUME = {56},
      YEAR = {2010},
    NUMBER = {1-2},
     PAGES = {3--47},
      ISSN = {0013-8584},
   MRCLASS = {17B40 (17B20)},
  MRNUMBER = {2674853},
MRREVIEWER = {Xin\ Tang},
       DOI = {10.4171/LEM/56-1-1},
       URL = {https://doi.org/10.4171/LEM/56-1-1},
}

@article {So1,
    AUTHOR = {Solleveld, Maarten},
     TITLE = {A local {L}anglands correspondence for unipotent
              representations},
   JOURNAL = {Amer. J. Math.},
  FJOURNAL = {American Journal of Mathematics},
    VOLUME = {145},
      YEAR = {2023},
    NUMBER = {3},
     PAGES = {673--719},
      ISSN = {0002-9327,1080-6377},
   MRCLASS = {22E50 (20G25)},
  MRNUMBER = {4596175},
MRREVIEWER = {Petar\ Baki\'{c}},
}

@article {Som1,
    AUTHOR = {Sommers, Eric},
     TITLE = {A generalization of the {B}ala-{C}arter theorem for nilpotent
              orbits},
   JOURNAL = {Internat. Math. Res. Notices},
  FJOURNAL = {International Mathematics Research Notices},
      YEAR = {1998},
    NUMBER = {11},
     PAGES = {539--562},
      ISSN = {1073-7928,1687-0247},
   MRCLASS = {20G20 (17B20)},
  MRNUMBER = {1631769},
MRREVIEWER = {R.\ W.\ Carter},
       DOI = {10.1155/S107379289800035X},
       URL = {https://doi.org/10.1155/S107379289800035X},
}

@article {Wa,
    AUTHOR = {Waldspurger, J.-L.},
     TITLE = {Produit scalaire elliptique},
   JOURNAL = {Adv. Math.},
  FJOURNAL = {Advances in Mathematics},
    VOLUME = {210},
      YEAR = {2007},
    NUMBER = {2},
     PAGES = {607--634},
      ISSN = {0001-8708,1090-2082},
   MRCLASS = {20G05},
  MRNUMBER = {2303234},
MRREVIEWER = {Bhama\ Srinivasan},
       DOI = {10.1016/j.aim.2006.07.005},
       URL = {https://doi.org/10.1016/j.aim.2006.07.005},
}

@article {Wa2,
    AUTHOR = {Waldspurger, J.-L.},
     TITLE = {Repr\'esentations de r\'eduction unipotente pour ${SO}(2n + 1)$, {I}: une involution},
   JOURNAL = {J. Lie Theory},
    VOLUME = {28},
      YEAR = {2018},
     PAGES = {381--426}
}

\end{document}